\documentclass[mathpazo]{cmr}
\usepackage{xcolor}
\usepackage{bm}

 \newtheorem{thm}{Theorem}[section]
 
 \newtheorem{lem}[thm]{Lemma}
 \newtheorem{prop}[thm]{Proposition}
 \theoremstyle{definition}
 
 \theoremstyle{remark}
 \newtheorem{rem}[thm]{Remark}
 
 \numberwithin{equation}{section}
\DeclareMathOperator*\dif{\mathop{}\!\mathrm{d}}

\newcommand{\R}{\mathbb{R}}

\begin{document}

\title{Threshold solutions for nonlocal reaction diffusion equations}

\author[H. Zhang, Y. Li and X. Yang]{He Zhang\affil{1,3},
                                        Yong Li\affil{1,2}\comma\corrauth~and Xue Yang\affil{1,2} }
\address{\affilnum{1}\ School of Mathematics,
                       Jilin University,
                       Changchun, Jilin, 130012, P.R. China.\\
         \affilnum{2}\ School of Mathematics and Statistics,
                       Center for Mathematics and Interdisciplinary Sciences,
                       Northeast Normal University,
                       Changchun, 130024, P.R. China.\\
         \affilnum{3}\ Tianyuan Mathematical Center in Northeast China,
                       Jilin University,
                       Changchun, Jilin, 130012, P.R. China.\\
}

\emails{{\tt he\_zhang@jlu.edu.cn} (H.~Zhang),
        {\tt liyong@jlu.edu.cn}  (Y.~Li),
        {\tt xueyang@jlu.edu.cn}  (X.~Yang)  }

\begin{abstract}
We study the Cauchy problem for nonlocal reaction diffusion equations with bistable nonlinearity in 1D spatial domain and investigate the asymptotic behaviors of solutions with a one-parameter family of monotonically increasing and compactly supported initial data. We show that for small values of the parameter the corresponding solutions decay to 0, while for large values the related solutions converge to 1 uniformly on compacts. Moreover, we prove that the transition from extinction (converging to 0) to propagation (converging to 1) is sharp. Numerical results are provided to verify the theoretical results.
\end{abstract}


\maketitle

\section{Introduction}
Reaction-diffusion equation
\begin{equation}\label{17}
  u_t=D\triangle u+ru(1-u),
\end{equation}
was first proposed by Fisher \cite{fisher1937wave} and  Kolmogorov, Petrovskii and Piskunov \cite{kolmogorov1937etude} in 1937 to describe the spatial spread of an advantageous allele in a given species (here $D$ is the diffusion constant and $r$ is the growth rate of the species).

 Since then, reaction-diffusion equations have attracted considerable attention, as they have proved to give an accurate description of various natural phenomena, ranging from heat and mass transfer, chemical kinetics, population dynamics, neuroscience and biomedical processes to spatial ecology. For more details on reaction-diffusion equations and their applications, the interested reader is
referred to  \cite{volpert2014elliptic} and the references therein.

In this paper, we  consider the following nonlocal reaction-diffusion equation
\begin{equation}\label{1}
   u_t=J*u-u-f(u),~x\in\mathbb{R},~t\geq0,
\end{equation}
where $J*u=\int_{\mathbb{R}}J(x-y)u(t,y)\,{\dif}y$.  This type of equations arises in applications where there exists nonlocal interaction between the variables involved. In population dynamics models, for instance, there are many situations when there exists long-range mechanism (such as the individuals are competing for a resource which
can redistribute itself, etc.) which is important in the process of evolution described (cf.~\cite{furter1989local,Gourley2000,coville2007non}). In this context, the value of the population density $u$ at a certain point $x$ may depend explicitly not only on the value of the function at $x$ but also on the values at other points: let $J(x-y)$ be the probability distribution of rates jumping from location $y$ to location $x$, then $J*u$ is the rate at which individuals are arriving at position $x$ from all other places and $-u(t,x)$ is the rate at which they are leaving location $x$ to all other locations (cf.~\cite{fife2003some}).

A number of attempts were made to use nonlocal reaction-diffusion models similar to \eqref{1} in other fields such as material science \cite{CHMAJ2000135},  solid phase transitions \cite{bates1997traveling,duncan_grinfeld_stoleriu_2000,0951-7715-14-2-303,orlandi1997travelling}, and finite scale microstructures in nonlocal elasticity \cite{Ren2000}. In particular, the relationship between \eqref{1}
and lattice models
\begin{equation}\label{37}
  u_t=D[u(x+1,t)-2u(x,t)+u(x-1,t)]-f(u).
\end{equation}
has been explored in \cite{bates1999discrete}.

There is an extensive literature on other problems involving nonlocal terms of a different type.
See, for instance, \cite{orlandi1997travelling,de1995travelling,DeMasi1993} for the continuum Ising model
\[
  u_t=\tanh(\beta J*u-h)-u,
\]
where $\beta>1$ and $h$ are constants, and \cite{caffarelli2010variational,Caffarelli2008,CPA21379,Chen2006} for nonlocal models involving the fractional Laplacian.

Throughout this paper, we make the following assumptions:\\
\textbf{(H1)} $J\in C^1(\mathbb{R}),~J\geq0,~ J(x)=J(-x)$~{for all}~$x$,~$\int_{\mathbb{R}}J\,{\dif} x=1$,~$J'\in L^1(\mathbb{R})$,~$\int_{\mathbb{R}}J(x)|x|\,{\dif} x<\infty$.\\
\textbf{(H2)} $f\in C^2(\mathbb{R})$,~$f(s)>0$ for $0<s<\alpha$, $f(s)<0$ for $\alpha<s<1$, $f(0)=f(1)=0,~f'(0)>0,~f'(1)>0$,~$1+f'(s)>0$~for $s\in(0,1)$,~$f$~{has only one zero}~$\alpha$~{in}~$(0,1)$,~{and no zeros outside}~$[0,1]$.

A prototypical example of such function is $f(s)=s(s-1)(s-\alpha)$ with $\alpha\in(0,1)$. Such equations play a fundamental role in population ecology. In addition, we assume that\\
\textbf{(H3)} $\int_0^1 f(s)\,{\dif} s<0$, and define $\beta$ by $\int_0^{\beta} f(s)\,{\dif} s=0.$

%

We should mention that the Laplacian operator $\triangle u$ in the classical reaction-diffusion equation
\begin{equation}\label{71}
  u_t=\triangle u
\end{equation}
can also be written as the $L^2$ gradient flow of the energy functional
\begin{equation}
  E_0[u]:=\frac{1}{2}\int_{\mathbb{R}}|\nabla u|^2\,{\dif} x,
\end{equation}
which measures how much $u$ deviates from being constant. Then $E_0[u] \geq 0$ and $E_0[u] = 0$ if and only if $u\equiv$ constant.

Consider the similar functional
\begin{equation}\label{36}
E_1[u]:=\frac{1}{4}\iint_{\mathbb{R}^2}J(x-y){\big(u(t,x)-u(t,y)\big)}^2\,{\dif} x{\dif} y.
\end{equation}
It is easy to check that $E_1[u]\geq0$ and $E_1[u] = 0$ only
when $u\equiv$ constant. Moreover, if $J$ satisfies \textbf{(H1)}, then $J*u-u$ is the $L^2$ gradient flow of \eqref{36}. In the sense that $E_0$ is analogous to $E_1$, the operator $J*u-u$ shares some essential properties with the Laplacian operator $\triangle u$, such as a form of the maximum principle (cf.~\cite{bates1997traveling,fife1998convolution,orlandi1997travelling}).

When we incorporate the nonlinear term $f(u)$, it is easy to check that \eqref{1} is the $L^2$ gradient flow of the energy functional
\begin{equation}
  E[u]:=\frac{1}{4}\iint_{\mathbb{R}^2}J(x-y){\big(u(t,x)-u(t,y)\big)}^2\,{\dif} x{\dif} y+\int_{\mathbb{R}} F(u)\,{\dif} x,
\end{equation}
where $F(u)=\int_0^uf(s)\,{\dif}s$.

There has been substantial research effort devoted to nonlocal reaction-diffusion equations of this kind. The existence, uniqueness and stability of monotone travelling wave solutions of \eqref{1} with bistable nonlinearity in one space dimension was investigated by Bates~ \cal{et~al} \cite{bates1997traveling}, Carr-Chmaj \cite{carr2004uniqueness}, Covillel-Dupaigne \cite{coville2007non} and Sun \cal{et~al} \cite{sun2011entire}. For higher dimensions, properties of solutions of nonlocal reaction diffusion equations with bistable nonlinearities were studied by Fife-Wang \cite{fife1998convolution}. Deriving the properties of solutions of non-autonomous nonlocal reaction diffusion equations is another challenge. We refer the readers to \cite{lim2015transition} and references therein for works of the reaction diffusion equations with spatially inhomogeneous nonlinearities $f=f(x,u)$.

In this paper, we study the asymptotic behaviors of solutions of the Cauchy problem for \eqref{1} with initial value
\begin{equation}\label{16}
u(0,x)=1_{[-L,L]}(x).
\end{equation}
The system \eqref{1}-\eqref{16} exhibits some interesting dynamics. By \textbf{(H2)}, the system has two constant steady states, 0 and 1. We are interested in the size of domains of attraction of the two constant steady states and the transition from one to the other when the initial conditions are varied. We are also interested in whether there exists some initial values for which the corresponding solutions (referred to as threshold solutions) do not converge to 0 or 1.

Threshold solution problems are difficult even for the classical diffusion. They were first put forth by Kanel' in \cite{kanel1964stabilization} for ignition nonlinearities in the context of reaction diffusion equation
\begin{equation}\label{38}
  u_t=u_{xx}+f(u),
\end{equation}
with the initial condition \eqref{16}, where $x\in\mathbb{R}$ and the nonlinearity $f$ is Lipschitz and satisfies \textbf{(H2)}. Kanel' proved that there exists some $L_0$ and $L_1$ such that if $L<L_0$, then $u(t,x)$ decays to 0 as $t\rightarrow\infty$; and if $L>L_1$, then the solution $u(t,x)$ converges to 1 as $t\rightarrow\infty$ uniformly on compacts. Recently, Zlato\v{s} \cite{zlatovs2006sharp} proved that $L_0=L_1$ and the transition is sharp: there is exactly one critical value (threshold value), and if the initial datum equals the threshold value, $u(t,x)$ converges to some stationary waves of \eqref{17}. Du and Matano \cite{du2010convergence} extended the results of \cite{zlatovs2006sharp} to more general families of initial data. In \cite{polacik2011threshold}, Pol{\'a}\v{c}ik addressed the question for nonautonomous parabolic equations on $\mathbb{R}^N$.

In this paper, we prove the similar results for \eqref{1}. To the best of our knowledge, this is the first study of threshold solutions in the context of nonlocal reaction-diffusion equations. We state our results in the following theorem, which will be proved in Section 3.
\begin{thm}\label{main}
Suppose that $J$ and $f$ satisfy \textbf{(H1)} and \textbf{(H2)}. Let $u^L\colon[0,\infty)\times\mathbb{R}\rightarrow[0,1]$ be the solution of \eqref{1} and \eqref{16}. Then there exists some $L^*>0$ such that

(1) if $L<L^*$, then $u^L\rightarrow0$ uniformly on $\mathbb{R}$ as $t\rightarrow\infty$;

(2) if $L=L^*$, then $u^{L^*}(t_n,\cdot)$ has a subsequence that converges to some $U$ uniformly in $x$ as $t_n\rightarrow\infty$, where $U$ is some solution of
\begin{equation}\label{8}
\begin{split}
   & J*U-U+f(U)=0 \\
    & U(0)=\beta;
\end{split}
\end{equation}

(3) if $L>L^*$, then $u^L\rightarrow1$ uniformly on compacts as $t\rightarrow\infty$.
\end{thm}

The rest of the paper is organised as follows. We devote Section 2 to a summary of relevant results, including the strong maximum principle and comparison principle, and some basic lemmas. In Section 3, we investigate the asymptotic behaviors of solutions depending on initial data and prove the main result Theorem \ref{main}. Numerical experiments are provided in Section 4 to illustrate our theoretical results. Some concluding remarks are presented in the last section.

\section{Preliminaries}
In this section, we state and prove some lemmas that we will use later.
We begin by considering the following initial value problem (IVP):
\begin{equation}\label{53}
  \begin{split}
     & u_t=J*u-u-f(u),~x\in\R,~t>0,\\
     & u(0,x)=u_0(x),~x\in\R.
  \end{split}
\end{equation}

\begin{lem}[Local well-posedness]  \label{exist}
Let $J$ and $f$ satisfy \textbf{(H1)} and \textbf{(H2)}. For any $0\leq u_0(x)\leq 1$, IVP \eqref{53} has a unique solution $u(t,x;u_0)$ which depends continuously on the initial condition $u_0(x)$. Moreover, $u(t,x;u_0)\in C([0,\infty)\times\R)$, if $u_0(x)\in C(\R)$.
\end{lem}

\begin{proof}
This closely follows the proof of Theorem 2.3 from \cite{sun2011entire}. We write IVP \eqref{53} in integral form as
\begin{equation*}
  u(t,x)=u_0(x)e^{-\mu t}+\int_{0}^te^{-\mu (t-s)}(J*u-u-f(u)+\mu u)\,{\dif}s,
\end{equation*}
where $\mu=\max\limits_{u\in[0,1]}|f'(u)|+1$. The local existence and uniqueness follow from the contraction mapping theorem.

 Now we prove the continuous dependence of $u(t,x;u_0)$ on $u_0$. Let $u(t,x;u_0)$ and $v(t,x;v_0)$ be the solution of IVP \eqref{53} with initial values $u(0,x)=u_0$ and $v(0,x)=v_0$, respectively. Then
\begin{equation*}
\begin{split}
  &|u(t,x)-v(t,x)|\leq |u_0-v_0|e^{-\mu t}+(\mu+2+M)\int_0^t e^{-\mu (t-s)}\|u(s,\cdot)-v(s,\cdot)\|_{L^{\infty}(\mathbb{R})}\,{\dif}s,
  \end{split}
\end{equation*}
where $M:=\max\limits_{u\in[0,1]}|f'(u)|$. The continuous dependence can be easily proved by Gronwall's inequality.
\end{proof}

The following lemma was first proved by F. Chen \cite{Chen2002807}.
\begin{lem}\label{max}
Let $J$ satisfy \textbf{(H1)}. Assume that, for some $T>0$, $u\in C^1([0,T],L^{\infty}(\mathbb{R}))$ is the solution of
\begin{equation}\label{56}
  \begin{split}
     & u_t\geq J*u-u-h(t,x)u,~(t,x)\in(0,T]\times\mathbb{R},\\
     & u(0,x)=u_0(x),x\in\mathbb{R},\\
  \end{split}
\end{equation}
 where $h(t,x)\in L^{\infty}([0,T]\times\mathbb{R})$. If $u_0(x)\geq0$ for almost all $x\in\mathbb{R}$, then $u(t,x)\geq0$ for almost all $x\in\mathbb{R}$ and $t\in[0,T]$. Moreover, if $u(t,x)$ is bounded and uniformly continuous on $(t,x)\in[0,T]\times\mathbb{R}$, then either $u\equiv0$ or $u>0$ on $(0,T]\times\mathbb{R}$.
\end{lem}
\begin{proof}
If $u_0(x)\geq0$ for almost all $x\in\mathbb{R}$, then we deduce from the continuity that there exists some small $t_0>0$ such that $u(t_0,x)\geq0$ for almost all $x\in\mathbb{R}$. Were the assertion false, there would exist a subset $U\subset\mathbb{R}$ of positive measure such that $u(t_0,x)=0$ and $u_t(t_0,x)<0$ for all $x\in U$. Then for $(t_0,x_0)\in(0,T]\times U$, we have by \eqref{56}
\begin{equation*}
\begin{split}
  &\frac{\partial u}{\partial t}(t_0,x_0)\geq\int_{\mathbb{R}}J(x_0-y)u(t_0,y)\,{\dif}y-u(t_0,x_0)-h(t_0,x_0)u(t_0,x_0)\\
&=\int_{\mathbb{R}}J(x_0-y)u(t_0,y)\,{\dif}y\geq 0.
\end{split}
\end{equation*}
 A contradiction arises. Therefore, if $u_0(x)\geq0$ for almost all $x\in\mathbb{R}$, then $u(t,x)\geq0$ for almost all $x\in\mathbb{R}$ and $t\in[0,T]$.

Suppose now $u$ is bounded and uniformly continuous and $u\not\equiv0$. If there exists some point $(t_1,x_1)\in(0,T]\times\mathbb{R}$ such that $u(t_1,x_1)=0$, then $(t_1,x_1)$ is a minimum point. It follows from \eqref{56} that $\int_{\mathbb{R}}J(x_1-y)u(t_1,y)\,{\dif}y=0$. Thus $u(t_1,x)\equiv0$ for all $x\in\mathbb{R}$. This completes the proof.
\end{proof}

Define sub- and super-solutions as in the theory of parabolic equations, then we have

\begin{lem}[Comparison Principle]\label{com}
Let $J$ satisfy \textbf{(H1)}. Suppose that $u_1(t,x)$ and $u_2(t,x)$ are a supersolution and a subsolution of \eqref{53}, respectively, with $u_1(0,x)\geq u_2(0,x)$, for all $x\in\mathbb{R}$. Then $u_1(t,x)\geq u_2(t,x)$ for all $x\in\mathbb{R}$ and $t>0$. Moreover, if $u_1(t,x)\not\equiv u_2(t,x)$ for $t>0$, then $u_1(t,x)>u_2(t,x)$ for all $x\in\mathbb{R}$ and $t>0$.
\end{lem}
\begin{proof}
Apply Lemma \ref{max} to $u_1-u_2$.
\end{proof}

Combining Lemmas 2.1-2.3, we obtain the following result.
\begin{lem}\label{prop1}
Let $J$ and $f$ satisfy \textbf{(H1)} and \textbf{(H2)}. Suppose that $u(t,x)$ is the solution of IVP \eqref{53} with $u_0(x)=1_{[-L,L]}(x)$. Then for fixed $t>0$, $u$ is  decreasing in $|x|$, i.e.~
\begin{equation}\label{12}
u(t,x)\geq u(t,y),~\mbox{if}~ |x|\leq |y|.
\end{equation}
\end{lem}

\begin{proof}
By symmetry, we only need to show $u(t,x)\geq u(t,y)$~{if}~ $0<x\leq y.$ Let us now consider IVP \eqref{53} with $u(0,x)=\chi_\varepsilon(x)$, where $\{\chi_\varepsilon(x)\}_{\varepsilon\geq0}$ is a family of smooth, nonnegative, symmetric functions. Assume further that $\chi_\varepsilon(x)$ is decreasing in $|x|$ for each $\varepsilon$ and converges to $1_{[-L,L]}(x)$ in $L^1(\mathbb{R})$ as $\varepsilon\rightarrow0$. Thanks to Lemma \ref{exist}, such IVP has a unique continuous solution $u^\varepsilon(t,x)$. Let $D_x^h u^\varepsilon(t,x):=\frac{1}{h}[u^\varepsilon(t,x+h)-u^\varepsilon(t,x)]$ be the difference quotient of $u^\varepsilon$ in the $x$-variable of size $h$ for any $x\in\mathbb{R}$, $t>0$, and $h>0$. Then $D^h_x  u^\varepsilon(0,x)=\frac{1}{h}\big[\chi_\varepsilon(x+h)-\chi_\varepsilon(x)\big]\leq 0$ for $x>0$. On the other hand, we have by \eqref{53}
\begin{equation}
\begin{split}
   & \frac{\partial}{\partial t}\bigg(D^h_x u^\varepsilon(t,x)\bigg)=\frac{1}{h}\bigg[\frac{\partial u^\varepsilon}{\partial t}(t,x+h)-\frac{\partial u^\varepsilon}{\partial t}(t,x)\bigg]\\
   &=\frac{1}{h}\bigg[\int_{\mathbb{R}}J(y)\big[u^\varepsilon(t,x+h-y)-u^\varepsilon(t,x-y)\big]\,{\dif} y\\
   &-(u^\varepsilon(t,x+h)-u^\varepsilon(t,x))-(f(u^\varepsilon(t,x+h))-f(u^\varepsilon(t,x)))\bigg]\\
   &=\int_{\mathbb{R}}J(y)D^h_x u^\varepsilon(t,x+h)\,{\dif} y-D^h_x u^\varepsilon(t,x)-f'(u^\varepsilon_\theta)D^h_x u^\varepsilon(t,x),
\end{split}
\end{equation}
where $u^\varepsilon_\theta$ is between $u^\varepsilon(t,x+h)$ and $u^\varepsilon(t,x)$. It then follows from Lemmas \ref{max} and \ref{com} that $D^h_x u^\varepsilon(t,x)\leq 0$ for $x>0$. Let $\varepsilon$ tend to 0 for any fixed $t$, we have $u^\varepsilon(t,x)\rightarrow u(t,x)$ uniformly in $x$ by Lemma \ref{exist}. The proof is complete.
\end{proof}

\begin{lem}\label{prop2}
Let $J$ and $f$ satisfy \textbf{(H1)} and \textbf{(H2)}. Suppose that $u(t,x)$ is the solution of IVP \eqref{53} with $u_0(x)=1_{[-L,L]}(x)$. Then there is a $t^*>0$ (possibly infinite) such that $u(t,0)$ as a function of $t$ is non-increasing on $[0,t^*)$ and non-decreasing on $[t^*,\infty)$.
\end{lem}

\begin{proof}
We wish to investigate the sign of the difference
\begin{equation*}
  D_t^h u(t,x):=u(t+h,x)-u(t,x)
\end{equation*}
for $x\in\mathbb{R}$, $t>0$, $h>0$. By use of the mean value theorem, $D_t^h u$ satisfies
\begin{equation*}
  \frac{\partial}{\partial t}\left({D}^h_t u\right)=J*{D}^h_t u(t,x)-{D}^h_t u(t,x)-f'(u_\theta){D}^h_t u(t,x),
\end{equation*}
where $u_\theta$ is between $u(t,x+h)$ and $u(t,x)$.
Set $\mathcal{A}:=\{(t,x)\colon{D}^h_t u(t,x)\leq0\}$. Then it follows from Lemmas \ref{max} and \ref{com} that $\mathcal{A}\cap(\{0\}\times\R)=\{0\}\times[-L,L]$. Indeed,
\begin{equation}\label{19}
  {D}^h_t u(0,x)=\left\{
\begin{array}{cc}
    u(h,x)-1\leq 0, & ~\mbox{if}~|x|\leq L; \cr
    u(h,x)\geq0, & ~\mbox{if}~|x|> L.
\end{array}\right.
\end{equation}
Thus by Lemmas \ref{max} and \ref{com}, there exists some $0<t^h\leq\infty$ such that ${D}^h_t u(t,0)\leq0$ for $0\leq t<t^h$, and ${D}^h_t u(t,0)\geq0$ for $t\geq t^h$. On the other hand, we can rewrite ${D}^h_t u(t,0)$ as
\begin{equation*}
  {D}^h_t u(t,0)={D}^{h/2}_t u(t+\frac{h}{2},0)+{D}^{h/2}_t u(t,0).
\end{equation*}
Then we have ${D}^h_t u(t,0)\geq0$ for any $t\geq t^{h/2}$, which implies $t^h\leq t^{h/2}$. Quite similarly, we can prove $t^{h/2}-\frac{h}{2}\leq t^h$. We have thus $t^h\leq t^{\frac{h}{2}}\leq t^h+\frac{h}{2}$, and Lemma \ref{prop2} follows with $t^*=\lim\limits_{n\rightarrow\infty}t^{\frac{h}{2^n}}$.
\end{proof}

\begin{lem}\label{lem5}
Suppose \textbf{(H1)}-\textbf{(H3)} hold. Let $u\colon[0,\infty)\times\mathbb{R}\rightarrow[0,1]$ be the solution of IVP \eqref{53} with $u_0(x)=1_{[-L,L]}(x)$. Then for every $\epsilon>0$, there is a $\delta>0$, depends only on $\epsilon$, such that if $x,y\in\mathbb{R}$ with $|x-y|<\delta$, we have
\begin{equation}
 |u(t,y)-u(t,x)|< \epsilon,~ \text{for any}~ t>0.
\end{equation}
\end{lem}
\begin{proof}
By Lemmas \ref{max}-\ref{com} and $0\leq u_0(x)\leq 1$, we have $0\leq u(t,x)\leq 1$ for any $(t,x)\in[0,\infty)\times\mathbb{R}$. For any $x,y\in\mathbb{R}$, define $\mathcal{D}u(t,x,y):=u(t,y)-u(t,x)$. Then from \eqref{53}, we obtain
\begin{equation}
\begin{split}
   & \frac{\partial}{\partial t}\mathcal{D}u(t,x,y)=\int_{\mathbb{R}}\big[J(y-z)-J(x-z)\big] u(t,z)\,{\dif} z\\
   &-\mathcal{D}u(t,x,y)-[f(u(t,y))-f(u(t,x))]\\
   &=\int_{\mathbb{R}}\big[J(y-z)-J(x-z)\big] u(t,z)\,{\dif} z-(1+f'(u_\theta))\mathcal{D}u(t,x,y),
\end{split}
\end{equation}
where $u_\theta$ is between $u(t,y)$ and $u(t,x)$. Since $J'(x)\in L^1(\mathbb{R})$, we can find a constant $L_1>0$ such that
\begin{equation*}
  \begin{split}
      &\int_{\mathbb{R}}|J(y-z)-J(x-z)|\,{\dif} z=|y-x|\int_{\mathbb{R}}\left|\int_0^1J'(x+\mu (y-x)-z)\,{\dif}\mu\right| \,{\dif} z\\
&\leq |y-x|\int_{\mathbb{R}}\int_0^1|J'(x+\mu (y-x)-z)|\,{\dif}\mu \,{\dif} z\leq L_1|y-x|.
  \end{split}
\end{equation*}
Note that $\mathcal{D}u(0,x,y)=1_{[-L,L]}(y)-1_{[-L,L]}(x)=2|y-x|$. Let $v(t)$ be the solution of
\begin{equation}\label{31}
  \begin{aligned}
     & v'(t)=L_1|y-x|-L_0 v(t),~t>0,\cr
      & v(0)=\mathcal{D}u(0,x,y),
  \end{aligned}
\end{equation}
where $L_0:=\inf\limits_{0\leq u\leq 1}(1+f'(u))$. By \textbf{(H2)}, $L_0>0$. Then
\begin{equation*}
  0<v(t)=e^{-L_0t}\mathcal{D}u(0,x,y)+\frac{L_1}{L_0}|y-x|(1-e^{-L_0t})\leq (2+\frac{L_1}{L_0})|y-x|.
\end{equation*}
Since $\mathcal{D}u(t,x,y)$ satisfies
\begin{equation*}\label{32}
\begin{split}
   & \frac{\partial}{\partial t}\mathcal{D}u(t,x,y)\leq L_1|y-x|-L_0\mathcal{D}u(t,x,y),~t>0,\\
   &\mathcal{D}u(0,x,y)=\mathcal{D}u(0,x,y),
\end{split}
\end{equation*}
thus we have by a comparison between \eqref{31} and \eqref{32} that
\[
  |u(t,y)-u(t,x)|=|\mathcal{D}u(t,x,y)|\leq v(t)\leq (2+\frac{L_1}{L_0})|y-x|.
\]
Thus for every $\epsilon>0$, there exists a $\delta=\frac{\epsilon}{2+\frac{L_1}{L_0}}>0$ such that $|x-y|<\delta$ implies that $|u(t,y)-u(t,x)|<\epsilon$, uniformly in $t>0$.
\end{proof}

The following lemmas deal with the asymptotic behavior as $t\rightarrow\infty$ of solutions of \eqref{1}. We focus on the situations in which a solution approaches a monotonic travelling front or a combination of two travelling fronts moving in opposite directions, uniformly in $x$ and exponentially in $t$ as $t\rightarrow\infty$.

Throughout this paper, a {travelling front} of \eqref{1} always refer to a pair $(U,c)$, where $U=U(\xi)\in[0,1]$ is a function on $\mathbb{R}$ and $c$ is a constant, such that $u(t,x)=U(x-ct):=U(\xi)$ is a solution of \eqref{1}, $U'(\xi)>0$ for finite $\xi=x-ct$ and
\[\lim\limits_{\xi\rightarrow-\infty}U(\xi)=0,~~\lim\limits_{\xi\rightarrow\infty}U(\xi)=1.\]
Such a function $U(\xi)$ satisfies the ordinary differential equation
\begin{equation}\label{63}
  cU'+J*U-U-f(U)=0.
\end{equation}
Multiplying both sides of \eqref{63} by $U'$ and integrating over $\mathbb{R}$ with respect to $\xi$, we have
\begin{equation}\label{64}
  c\int_{\mathbb{R}}(U')^2\,{\dif}\xi=\int_{\mathbb{R}} f(U)\,{\dif}U=\int_0^1 f(U)\,{\dif}U.
\end{equation}
Then it follows from assumption \textbf{(H3)} that
\begin{equation*}
  c<0.
\end{equation*}

\begin{rem}\label{rem1}
The existence and uniqueness up to translation of travelling front of \eqref{1} has been extensively studied. We refer interested readers to \cite{bates1997traveling,sun2011entire} and the references therein.
\end{rem}

By Theorem 3.1 in \cite{sun2011entire}, the following inequalities hold.

\begin{lem}\label{lem9}
Suppose \textbf{(H1)}-\textbf{(H3)} hold. $(U,c)$ is a travelling front of \eqref{1} satisfying $U'(\xi)>0$. Then there exist constants $\lambda_1<0$ and $\lambda_2>0$ such that
\begin{equation}
   \begin{split}
       & ke^{\lambda_1\xi}\leq 1-U(\xi)\leq Ke^{\lambda_1\xi},~\text{for}~\xi\geq0, \\
       & \bar{k}e^{\lambda_2\xi}\leq U(\xi)\leq \bar{K}e^{\lambda_2\xi},~\text{for}~\xi\leq0.
   \end{split}
\end{equation}
Here $k,~K,~\bar{k}$ and $\bar{K}$ are positive constants.
\end{lem}

\begin{lem}\label{lem7}
Suppose \textbf{(H1)}-\textbf{(H3)} hold. Let $u\colon[0,\infty)\times\mathbb{R}\rightarrow[0,1]$ be the solution of \eqref{1}. Suppose that
\begin{equation*}
  0\leq u_0(x)\leq1,~\limsup\limits_{x\rightarrow-\infty}u_0(x)<\alpha,~~\liminf\limits_{x\rightarrow\infty}u_0(x)>\alpha.
\end{equation*}
Then there exist constants $x_1,~x_2,~\kappa>0,~\varepsilon_0>0$ such that
\begin{equation}\label{59}
  U(x-ct-x_1)-\varepsilon_0e^{-\kappa t}\leq u(t,x)\leq U(x-ct-x_2)+\varepsilon_0e^{-\kappa t}
\end{equation}
for all $x\in\mathbb{R}$ and $t>0$. Here $\alpha\in(0,1)$ is the same as in assumption \textbf{(H2)}, i.e. $f(\alpha)=0$. $(U,c)$ is a travelling front of \eqref{1}.
\end{lem}
\begin{proof}
We prove only the right-hand inequality, the other is similar. Functions $p(t)$ and $\varepsilon(t)$ will be chosen such that
\begin{equation}
  v(t,x):=\min\{1,U(x-ct-p(t))+\varepsilon(t)\}
\end{equation}
is a supersolution of IVP \eqref{53}. Let $\varepsilon_0$ be any number such that
\begin{equation*}
\limsup\limits_{x\rightarrow-\infty}u_0(x)<\varepsilon_0<\alpha.
\end{equation*}
Then take sufficiently large $x^*>0$ such that
\begin{equation*}
  U(\tilde{x}-x^*)+\varepsilon_0\geq u_0(\tilde{x})~\text{for all}~\tilde{x}:=x-ct.
\end{equation*}
Let
\begin{equation}
  \Phi(u,\varepsilon):=\left\{
                         \begin{array}{ll}
                           \frac{1}{\varepsilon}\big[f(u)-f(u+\varepsilon)\big], & \hbox{$\varepsilon>0$;} \\
                           -f'(u), & \hbox{$\varepsilon=0$.}
                         \end{array}
                       \right.
\end{equation}
Then $\Phi(u,\varepsilon)$ is continuous for $\varepsilon\geq0$, and for $0<\varepsilon\leq\varepsilon_0$ we have $0<\varepsilon\leq\varepsilon_0<\alpha$. It follows from \textbf{(H2)} that $\Phi(0,\varepsilon)=-\dfrac{f(\varepsilon)}{\varepsilon}<0$ and  $\Phi(0,0)=-f'(0)<0$. Thus we can find some $\kappa>0$ such that
\begin{equation*}
  \Phi(0,\varepsilon)\leq-2\kappa,~\text{for}~0\leq\varepsilon\leq\varepsilon_0.
\end{equation*}
By continuity, we can find some small $d>0$ so that
\begin{equation}\label{65}
  \Phi(u,\varepsilon)\leq-\kappa,~\text{for}~0\leq u\leq d,~0\leq\varepsilon\leq\varepsilon_0.
\end{equation}
In this range, we have
\begin{equation*}
  f(u)-f(u+\varepsilon)\leq -\kappa\varepsilon.
\end{equation*}
Setting $\eta(t):=x-ct-p(t)$, and using the fact that
\begin{equation*}
  cU'+J*U-U-f(U)=0,
\end{equation*}
we obtain that
\begin{equation*}
\begin{split}
   &v_t-J*v+v+f(v) \\
   &=(-c-p'(t))U'(\eta)+\varepsilon'(t)-J*U(\eta)+U(\eta)+f(U(\eta)+\varepsilon(t))\\
   &=-p'(t)U'(\eta)+\varepsilon'(t)-f(U(\eta))+f(U(\eta)+\varepsilon(t)),
\end{split}
\end{equation*}
if $v<1$ (noting the definition of $v$). Thus when $0\leq u\leq d,~0\leq\varepsilon\leq\varepsilon_0$, we have
\begin{equation*}
\begin{split}
   &v_t-J*v+v+f(v) \geq -p'(t)U'(\eta)+\varepsilon'(t)+\kappa\varepsilon(t)\geq +\varepsilon'(t)+\kappa\varepsilon(t),
\end{split}
\end{equation*}
provided $p'(t)<0$, since $U'>0$ (see the definition of $U$). We choose $\varepsilon(t)=\varepsilon_0 e^{-\kappa t}$, which results in
\begin{equation*}
  v_t-J*v+v+f(v)\geq 0.
\end{equation*}
By possibly further reducing the size of $\kappa$ and $d$ and using the same
arguments, we may also obtain that
\begin{equation*}
  v_t-J*v+v+f(v)\geq 0,~\text{when}~1-d\leq u\leq 1.
\end{equation*}

For intermediate values $d\leq u\leq 1-d$, by the monotonicity of $U$ and \textbf{(H2)}, we can find a $c_0>0$ and a $\kappa_1>0$ such that
\begin{equation}\label{66}
  U'\geq c_0~\text{and}~f(U+\varepsilon)-f(U)\geq -\kappa_1\varepsilon.
\end{equation}
Therefore, if $v<1$, then
\begin{equation*}
\begin{split}
   &v_t-J*v+v+f(v)\geq -c_0p'(t)-\kappa\varepsilon_0 e^{-\kappa t}-\kappa_1\varepsilon_0 e^{-\kappa t}.
\end{split}
\end{equation*}
We now set
\begin{equation*}
  p(t)=x^*-c_2+c_2e^{-\kappa t},
\end{equation*}
where $c_2=\dfrac{\varepsilon_0}{\kappa c_0}(\kappa+\kappa_1)$. Thus $p'(t)=-c_2\kappa e^{-\kappa t}=-\dfrac{\varepsilon_0}{ c_0}(\kappa+\kappa_1)e^{-\kappa t}<0$. Consequently,  $p(t)$ is decreasing and approaches a finite limit $x^*-c_2$ as $t\rightarrow\infty$. Thus
\begin{equation*}
  v_t-J*v+v+f(v)\geq 0
\end{equation*}
whenever $v<1$ in this range, and by our assumption on $x^*$, $v$ is a supersolution of \eqref{1} if $v<1$. Therefore
\begin{equation}
\begin{split}
  u(t,x)\leq v(t,x) &\leq U(x-ct-p(t))+\varepsilon_0e^{-\kappa t} \\
    & \leq U(x-ct-x^*+c_2)+\varepsilon_0e^{-\kappa t}
\end{split}
\end{equation}
Taking $x_2=x^*-c_2$, we complete the proof.
\end{proof}

\begin{lem}\label{lem4}
Suppose \textbf{(H1)}-\textbf{(H3)} hold. Let $u\colon[0,\infty)\times\mathbb{R}\rightarrow[0,1]$ be the solution of IVP \eqref{53} with $u_0(x)$ satisfying $0\leq u_0(x)\leq 1$, and
\begin{equation*}
  \limsup\limits_{|x|\rightarrow\infty}u_0(x)<\alpha,~u_0(x)>\alpha+\bar{\alpha},~\text{for}~|x|<\bar{L},
\end{equation*}
where $\bar{\alpha}$ and $\bar{L}$ are positive constants.  Then if $\bar{L}$ is sufficiently large (depending on $\bar{\alpha}$ and $f$), there exist constants $x_1,~x_2,\kappa>0$ and $\varepsilon_0>0$ such that
\begin{equation}\label{15}
\begin{aligned}
  &U(x-ct-x_1)+U(-x-ct-x_1)-1-\varepsilon_0 e^{-{\kappa} t}\leq u(t,x)\cr
  &\quad\leq U(x-ct-x_2)+U(-x-ct-x_2)-1+\varepsilon_0 e^{-{\kappa} t}
\end{aligned}
\end{equation}
for all $x\in\mathbb{R}$ and $t>0$.
\end{lem}
\begin{proof}
First we prove the right-hand inequality. We get from Lemma \ref{lem7} that there exist constants $x_2$, $\kappa_1>0$ and $\varepsilon_1>0$ such that
\begin{equation*}
u(t,x)\leq U(x-ct-x_2)+\varepsilon_1 e^{-\kappa_1 t}.
\end{equation*}
The same argument applied to $u(t,-x)$ leads to
\begin{equation}\label{67}
u(t,x)\leq U(-x-ct-\bar{x_2})+\bar{\varepsilon_1} e^{-\kappa_1 t}
\end{equation}
for some $\bar{x_2}$, $\bar{\varepsilon_1}>0$, and $\kappa_1>0$. Since decreasing $x_2$ and $\bar{x_2}$ and increasing $\varepsilon_1$ and $\bar{\varepsilon_1}$ strengthens the inequality, we may assume that $x_2=\bar{x_2}<0$ and $\varepsilon_1=\bar{\varepsilon_1}$. Hence
\begin{equation}\label{33}
  u(t,x)\leq\min\{U(x-ct-x_2),U(-x-ct-x_2)\}+\varepsilon_1 e^{-\kappa_1 t}.
\end{equation}
If $x\geq0$, the monotonicity of $U$ implies
\[
  U(x-ct-x_2)\geq U(-x-ct-x_2).
\]
Furthermore, by Lemma \ref{lem9}, there exists some $K>0$ and $\lambda_1<0$ such that
\begin{equation*}
  1-U(x-ct-x_2)\leq1-U(-ct-x_2)\leq K e^{\lambda_1(-ct-x_2)}.
\end{equation*}
Hence from \eqref{33}, for suitable $\varepsilon_0>\varepsilon_1$ and $-x_2$ large enough, we have
\[
  \begin{aligned}
     u(t,x) & \leq U(-x-ct-x_2)+\varepsilon_1 e^{-\kappa_1 t} \cr
       & \leq U(-x-ct-x_2)+\varepsilon_1 e^{-\kappa_1 t}+U(x-ct-x_2)-1+K e^{\lambda_1(-ct-x_2)}\cr
       & \leq U(-x-ct-x_2)+U(x-ct-x_2)-1+\varepsilon_0 e^{-\kappa_1 t}.
   \end{aligned}
\]
A similar argument may be used for $x\leq0$. Thus the right-hand inequality holds.

We now prove the left-hand inequality of \eqref{15}. Define
\begin{equation*}
{v}(t,x):=\max\{U(\zeta+)+U(\zeta-)-1-\varepsilon(t),0\},
\end{equation*}
where $U(\zeta+):=U(x-ct-\zeta(t))$, $U(\zeta-):=U(-x-ct-\zeta(t))$, for suitable $\varepsilon(t)>0$ and $\zeta(t)<0$ (with $\zeta'(t)>0$). We intend to prove that if $v>0$, $v$ is a subsolution. By \eqref{63}, if $v>0$, we easily obtain
\begin{equation*}
\begin{split}
   &v_t-J*v+v+f(v)\\
&=U'(\zeta+)(-c-\zeta'(t))+U'(\zeta-)(-c-\zeta'(t))-\varepsilon'(t)\\
&-J*(U'(\zeta_+)+U'(\zeta_-))+U'(\zeta_+)+U'(\zeta_-)+f\big(U(\zeta+)+U(\zeta-)-1-\varepsilon(t)\big)\\
& =-\zeta'(t)(U'(\zeta_+)+U'(\zeta_-))-\varepsilon'(t)\\
&-f(U(\zeta_+))-f(U(\zeta_-))+f\big(U(\zeta+)+U(\zeta-)-1-\varepsilon(t)\big).
\end{split}
\end{equation*}
Let $\varepsilon_0'$ and $\varepsilon_2$ be positive constants such that
\begin{equation*}
  \alpha<1-\varepsilon_2<1-\varepsilon_0'<\alpha+\bar{\alpha},
\end{equation*}
and let $d$ be as in the proof of Lemma \ref{lem7}.  We then see that for some $\kappa>0$,
\begin{equation*}
\begin{split}
   & f\big(U(\zeta+)+U(\zeta-)-1-\varepsilon(t)\big)-f(U(\zeta_-))\leq -\kappa(1-U(\zeta_+)+\varepsilon(t))
\end{split}
\end{equation*}
for $1-d\leq U(\zeta_-)\leq 1$, $0\leq 1-U(\zeta_+)+\varepsilon(t)\leq \varepsilon_2$. The latter inequality holds if $0\leq \varepsilon(t)\leq \varepsilon_0'$, $x\geq0$, and sufficiently large $-\zeta(t)$. Indeed, noting $c<0$, we obtain, by the monotonicity of $U$ and Lemma \ref{lem9}, that
\begin{equation*}
\begin{split}
1-U(\zeta_+)+\varepsilon(t)&\leq 1-U(-\zeta(t))+\varepsilon_0'\leq Ke^{\lambda_1(-\zeta(t))}+\varepsilon_0'\leq \varepsilon_2.
\end{split}
\end{equation*}
Finally note that
$$-f(U(\zeta_+))=f(1)-f(U(\zeta_+))\leq M(1-U(\zeta_+)),$$
 for $M=\max\limits_{\theta\in[0,1]}|f'(\theta)|>0.$ Therefore, for $1-d\leq U(\zeta_-)\leq 1$, $0\leq \varepsilon(t)\leq \varepsilon_0'$, $x\geq0$, sufficiently large $-\zeta(t)$ and $\zeta'(t)>0$, we have
\begin{equation*}
\begin{split}
   &v_t-J*v+v+f(v)\\
& =-\zeta'(t)(U'(\zeta_+)+U'(\zeta_-))-\varepsilon'(t)-f(U(\zeta_+))-f(U(\zeta_-))\\
&+f\big(U(\zeta+)+U(\zeta-)-1-\varepsilon(t)\big)\\
&\leq -\varepsilon'(t)+M(1-U(\zeta_+))-\kappa(1-U(\zeta_+)+\varepsilon(t))\\
&\leq (M-\kappa)(1-U(\zeta_+))-\varepsilon'(t)-\kappa\varepsilon(t)\\
&\leq (M-\kappa)Ke^{\lambda_1(-ct-\zeta(t))}-\varepsilon'(t)-\kappa\varepsilon(t).
\end{split}
\end{equation*}
Setting  $\varepsilon(t):=\varepsilon_0' e^{-\kappa_2 t}$ for $0<\kappa_2<\kappa$,
we obtain for the above range,
\begin{equation}
\begin{split}
   &v_t-J*v+v+f(v)\\
&\leq (M-\kappa)Ke^{\lambda_1(-ct-\zeta(t))}-(\kappa-\kappa_2)\varepsilon_0' e^{-\kappa_2 t}\leq0,
\end{split}
\end{equation}
provided $\kappa_2<\lambda_1c$ and $-\zeta$ large enough.

A similar argument holds for $0\leq U(\zeta_-)\leq d$, $0\leq \varepsilon(t)\leq \varepsilon_0'$, $x\geq0$, provided $v>0$. Finally for $d\leq U(\zeta_-)\leq 1-d$, $x\geq0$, we have
\begin{equation*}
   \begin{split}
       & f\big(U(\zeta+)+U(\zeta-)-1-\varepsilon(t)\big)-f(U(\zeta_-))\leq \kappa_3 \big(1-U(\zeta_+)+\varepsilon(t)\big), \\
       & U'(\zeta+)+U'(\zeta-)\geq 2c_0>0,
   \end{split}
\end{equation*}
for some $c_0>0$ and $\kappa_3>0$. Therefore,
\begin{equation*}
\begin{split}
&{v}_t-J*{v}+{v}+f({v})\\
& =-\zeta'(t)(U'(\zeta_+)+U'(\zeta_-))-\varepsilon'(t)-f(U(\zeta_+))-f(U(\zeta_-))\\
&+f\big(U(\zeta+)+U(\zeta-)-1-\varepsilon(t)\big)\\
&\leq -2c_0\zeta'(t)+(\kappa_3+M) Ke^{\lambda_1(-ct-\zeta(t))}+(\kappa_2+\kappa_3)\varepsilon_0' e^{-\kappa_2 t}.
\end{split}
\end{equation*}
Choose $\zeta(t)$ such that
\begin{equation*}
  -2c_0\zeta'(t)+(\kappa_3+M) Ke^{-c\lambda_1t}+(\kappa_2+\kappa_3)\varepsilon_0' e^{-\kappa_2 t}=0,
\end{equation*}
with $\zeta(0)$ sufficiently large and negative. Then from above we have
\begin{equation*}
{v}_t-J*{v}+{v}+f({v})\leq 0,
\end{equation*}
for all $x\geq0$ with $v>0$. A similar argument shows that this conclusion holds for $x\leq0$ as well. Now $v$ will be a subsolution if we can prove that $v(0,x)\leq u_0(x)$. We note that
\begin{equation*}
  v(0,x)=U(x-\zeta(0))+U(-x-\zeta(0))-1-\varepsilon_0'<1-\varepsilon_0'<\alpha+\bar{\alpha}\leq u_0(x)
\end{equation*}
for $|x|\leq \bar{L}$, and
\begin{equation*}
  v(0,x)\leq U(-L_0-\zeta(0))-\varepsilon_0'\leq 0\leq u_0(x)
\end{equation*}
for $|x|\geq L_0$, for some $L_0$ depending on $\zeta(0)$. Therefore if $\bar{L}\geq L_0$, we shall have $v(0,x)\leq u_0(x)$ for all $x$. Thus $v(t,x)>0$ is a subsolution to \eqref{53}. It now follows that
\begin{equation*}
  u(t,x)\geq v(t,x)\geq U(x-ct-\zeta(\infty))+U(-x-ct-\zeta(\infty))-1-\varepsilon_0'e^{-\kappa_2 t}.
\end{equation*}
Set $x_1=\zeta(\infty)$ and ${\kappa}=\min\{\kappa_1,\kappa_2\}$.
This completes the proof.
\end{proof}

\begin{lem}
There exit functions $\omega(\epsilon)$ and $T(\epsilon)$, defined for small positive $\epsilon$ and satisfying $\lim\limits_{\epsilon\downarrow0}\omega(\epsilon)=0$, such that if
\begin{equation}\label{2}
  |u(t_0,x)-U(x-ct_0-x_0)|<\epsilon
\end{equation}
for some $x_0$, $t_0>T(\epsilon)$, and all $x<0$, then
\begin{equation*}
  |u(t,x)-U(x-ct-x_0)|<\omega(\epsilon)
\end{equation*}
for all $t>t_0$ and $x<0$.
\end{lem}
\begin{proof}
Define $v(t,x)=\max\{U(x-ct-p(t))-\varepsilon_0e^{-\kappa t},0\}$, where $p(t)=c_1+c_2e^{-\kappa t}$. If $\kappa>0$ is sufficiently small and $c_2=c_{\kappa}\varepsilon_0$ for a certain constant $c_{\kappa}$ depending only on $\kappa$, then for arbitrary $c_1$ and $\varepsilon_0$, using the same procedure as in Lemma \ref{lem7}, we have
\begin{equation*}
  {v}_t-J*{v}+{v}+f({v})\leq 0, ~\text{if}~v>0.
\end{equation*}
From \eqref{2}, we have
\begin{equation*}
  u(t_0,x)\geq U(x-ct_0-x_0)-\epsilon.
\end{equation*}
If we now set $\varepsilon_0=\epsilon e^{\kappa t_0}$, $c_2=\epsilon c_{\kappa}e^{\kappa t_0}$, and $c_1=x_0-\epsilon c_{\kappa}$, then
\begin{equation*}
  v(t_0,x)=U(x-ct_0-x_0)-\epsilon\leq u(t_0,x).
\end{equation*}
From Lemma \ref{lem4} and Lemma \ref{lem9}, we have, for some $K>0$ and $\lambda_1<0$
\begin{equation*}
   \begin{split}
   u(t,0)&\geq 2U(-ct-x_1)-1-\varepsilon_0' e^{-{\kappa'} t}\\
   &\geq1-\varepsilon_0' e^{-{\kappa'} t}-2Ke^{\lambda_1(-ct-x_1)}.
   \end{split}
\end{equation*}
On the other hand, for $t\geq t_0$, if $v>0$,
\begin{equation*}
  v(t,0)=U(-ct-p(t))-\varepsilon_0e^{-\kappa t}<1-\varepsilon_0e^{-\kappa t}=1-\epsilon e^{\kappa t_0-\kappa t}.
\end{equation*}
Thus
\begin{equation}
   u(t,0)-v(t,0)\geq -\varepsilon_0' e^{-{\kappa'} t}-2Ke^{\lambda_1(-ct-x_1)}+\epsilon e^{\kappa t_0-\kappa t}.
\end{equation}
Choose $\kappa$ such that $0<\kappa<\kappa'$, $\kappa<\lambda_1c$. Thus
\begin{equation}
   u(t,0)-v(t,0)\geq [\epsilon-(\varepsilon_0'+2Ke^{-\lambda_1x_1})e^{-\kappa t_0}] e^{\kappa t_0-\kappa t}>0
\end{equation}
for sufficiently large $t_0$. Therefore, it follows from Lemmas \ref{max} and \ref{com} that $v(t,x)\leq u(t,x)$ in the region $x<0$ and $t\geq t_0$. That is to say, if $v>0$, then
\begin{equation*}
  u(t,x)\geq v(t,x)=U(x-ct-p(t))-\varepsilon_0e^{-\kappa t}\geq U(x-ct-x_0)-\omega(\epsilon),
\end{equation*}
for $x<0$ and $t\geq t_0$. A similar argument can be used to show that $u(t,x)\leq U(x-ct-x_0)+\omega(\epsilon)$. This completes the proof.
\end{proof}

%
%
%

\section{Proof of Theorem 1.1}
We can now prove our main result. Assume again that $J$ satisfies \textbf{(H1)}, and $f$ satisfies \textbf{(H2)}-\textbf{(H3)}. Let $u^L$ be the solution of the initial value problem
  \begin{align}
      &u_t=J*u-u-f(u),~x\in\R,~t\geq0, \label{57}\\
      &u(0,x)=1_{[-L,L]}(x),x\in\mathbb{R}. \label{58}
  \end{align}
Define two sets $M_0$ and $M_1$ as follows:
\begin{equation}
  M_0:=\{L>0\colon\lim_{t\rightarrow\infty} \|u^L(t,\cdot)\|_{L^{\infty}(\mathbb{R})}=0\},
\end{equation}
and
\begin{equation}
  M_1:=\{L>0\colon u^L(t,\cdot)\rightarrow 1,~\text{as}~ t\rightarrow\infty,~\text{uniformly on compact sets in}~\mathbb{R}\}.
\end{equation}
Let $T$ deonote the threshold set
\begin{equation}
  T:=(0,\infty)\setminus(M_0\cup M_1).
\end{equation}
 The comparison principle (Lemma \ref{com}) implies the threshold set $T$, if nonempty, lies between $M_0$ and $M_1$. The following lemma gives a necessary and sufficient condition for checking whether a value $L$ is an element of $M_0$.


\begin{prop}\label{prop3}
Let $\beta$ be such that $\int_0^{\beta} f(s)\,{\dif} s=0$. Then $L\in M_0$ if and only if $\lim\limits_{t\rightarrow\infty}u^L(t,0)<\beta$.
\end{prop}

\begin{proof}
It follows from Lemma \ref{com} and Lemma \ref{prop2} that $u_*^L:=\lim\limits_{t\rightarrow\infty}u^L(t,0)$ is well defined and non-decreasing in $L$. It suffices to show that if $u_*^L<\beta$, then $L\in M_0$. Suppose now $u_*^L<\beta$. Consequently, by Lemma 2.4 and Lemma 2.5, it is possible to find a $t_0$ such that for all $x\in\mathbb{R}$ and $t\geq t_0$,
\begin{equation*}
u^L(t,x)\leq \frac{1}{2}\big(u_*^L+\beta\big).
\end{equation*}
Given any $\varepsilon>0$, we can define a Lipschitz function $\tilde{f}\colon[0,1]\rightarrow\mathbb{R}$ such that
\begin{equation*}
\left\{
\begin{aligned}
  &\tilde{f}=0 ~\text{on}~ [0,\varepsilon],~\tilde{f}'(\varepsilon)>0, \\
  &\tilde{f}\leq f ~\text{on}~ \big(\varepsilon,\frac{1}{2}(u_*^L+\beta)\big] ~\text{and}~ \tilde{f}~ \text{has a single zero within this interval},\\
  &\tilde{f}<0 ~\text{on}~ \big(\frac{1}{2}(u_*^L+\beta),1\big), \tilde{f}(1)=0, \tilde{f}'(1)>0,\\
  &\int_0^1 \tilde{f}(\theta)\,{\dif} \theta>0.
  \end{aligned}
  \right.
\end{equation*}
Since $\tilde{f}\leq f$ on $\big(0,\frac{1}{2}(u_*^L+\beta)\big)$, starting from time $t_0$, we have
\begin{equation*}
  u^L_t\leq J*u^L-u^L-\tilde{f}(u^L),
\end{equation*}
that is, $u^L$ is a subsolution of the equation
\begin{equation}\label{9}
v_t= J*v-v-\tilde{f}(v).
\end{equation}
Let $v(t,x)=\phi(x-\nu t)$ be a travelling front of \eqref{9} satisfying $\phi(-\infty)=\varepsilon$ and $\phi(\infty)=1$. It follows by \eqref{64} that in this case $\nu>0$. Note that
\begin{equation*}
  u^L_t\leq J*u^L-u^L+Mu^L,
\end{equation*}
where $M=\max\limits_{s\in[0,1]}\{-f'(s)\}$. Hence by Lemma \ref{max} and Lemma \ref{com}, we have
\begin{equation*}
  u^L(t,x)\leq e^{Mt}w(t,x),
\end{equation*}
where $w$ is the solution of the following Cauchy problem:
\begin{equation*}
   \begin{split}
       & w_t=J*w-w, w(0,x)=u_0(x).
   \end{split}
\end{equation*}
Applying the Fourier transform to this equation, we have $\hat{w}(t,\xi)=e^{(\hat{J}(\xi)-1)t}\hat{u_0}(\xi).$ Taking the inverse Fourier transform, it follows from $J\in L^1(\mathbb{R})$ and the compactness of the support of $u_0(x)=1_{[-L,L]}(x)$ that, $w(t,x)\rightarrow0,~\text{as}~|x|\rightarrow\infty$, for fixed $t_0<\infty$. And so
\begin{equation*}
  u^L(t_0,x)\rightarrow0,~\text{as}~|x|\rightarrow\infty,~\text{for fixed}~t_0<\infty.
\end{equation*}
Consequently, for sufficiently large $x_0$, we have (recall that $\phi(\infty)=1$, $u^L(t_0,\cdot)\in(0,1)$)
\begin{equation*}
u^L(t_0,x)\leq \phi(x+x_0-\nu t_0).
\end{equation*}
Since $\phi$ is a solution and $u^L$ is a subsolution of \eqref{9}, it immediately follows that
\begin{equation*}
u^L(t,x)\leq \phi(x+x_0-\nu t)~\text{ for all}~ t\geq t_0.
\end{equation*}
Then $u^L(t,0)\leq \phi(x_0-\nu t)\rightarrow\varepsilon$, as $t\rightarrow\infty$ (note that $\nu>0$). Thus for any $\varepsilon>0$, $u_*^L\leq\varepsilon$. Therefore, $u_*^L=0$ and $L\in M_0$.
\end{proof}

\begin{prop}\label{prop7}
$M_0$ is open.
\end{prop}
\begin{proof}
It follows from Proposition \ref{prop3} that $L\in M_0$ if $L$ is small enough. Thus $M_0\neq\emptyset$. Set $L^0:=\sup M_0$, then $0<L^0\leq\infty$, and the comparison principle (Lemma 2.3) gives $(0,L^0)\subset M_0$. It remains to show that $L^0\notin M_0$ if $L^0<\infty$. Suppose the contrary, then given any $\varepsilon_0>0$, there would exist some $t_0>0$ such that
\begin{equation*}
u^{L_0}(t,\cdot)\leq\varepsilon_0 ~\text{for all}~ t\geq t_0.
\end{equation*}
By the continuous dependence of $u^L$ on initial value, for arbitrary small positive constant $\varepsilon$, we can find a $\delta>0$ sufficiently small such that
\begin{equation*}
|u^{\tilde{L}}(t,x)-u^{L_0}(t,x)|<\varepsilon,
\end{equation*}
for $\tilde{L}:=L_0+\delta$. Consequently, if $t\geq t_0$,
\begin{equation*}
u^{\tilde{L}}(t,x)\leq \varepsilon_0+\varepsilon,
\end{equation*}
which follows that $\lim\limits_{t\rightarrow\infty}u^{\tilde{L}}(t,x)=0$ and $\tilde{L}\in M_0$. This contradicts the definition of $L_0$.
\end{proof}

\begin{prop}\label{prop4}
$L\in M_1$ if and only if $\lim\limits_{t\rightarrow\infty}u^L(t,0)>\beta$.
\end{prop}
\begin{proof}
It suffices to show that if $\lim\limits_{t\rightarrow\infty}u^L(t,0)>\beta$, then $L\in M_1$. We first claim that $\lim\limits_{t\rightarrow\infty}u^L(t,0)=1$. Otherwise,  for any $\gamma\in(\lim\limits_{t\rightarrow\infty}u^L(t,0),1)$, we can define a Lipschitz function $g$ satisfying
\[
  \left\{
     \begin{array}{l}
      g(s)\geq f(s), ~ \hbox{for  $s\geq 0$;} \cr
       g(0)=g(\gamma^0)=g(\gamma)=0 ~\quad\hbox{where $\gamma^0\in(0,\gamma)$;} \cr
       g'(0)>0,~g'(\gamma)>0 ~ \hbox{and $\int_0^{\gamma}g(s)\,{\dif} s<0$ .}
     \end{array}
  \right.
\]
Let $v(t,x)$ be the solution of
\begin{equation}\label{20}
  v_t=J*v-v-g(v), t>0, x\in\mathbb{R}, v(0,x)=1_{[-L,L]}(x),
\end{equation}
and $\psi$ be the unique travelling front (with speed $\mu$) of \eqref{20} satisfying $\psi'>0$, $\psi(-\infty)=0$ and $\psi(\infty)=\gamma$. (note that by \eqref{64} $\mu<0$)
Applying Lemma \ref{lem4} to \eqref{20}, then there exist $x_1,~x_2$, $\varepsilon>0$, and $\kappa>0$ so that
\begin{equation*}
\begin{split}
  &\psi(x-\mu t-x_1)+\psi(-x-\mu t-x_1)-1-\varepsilon_0 e^{-{\kappa} t}\leq v(t,x)\cr
  &\quad\leq \psi(x-\mu t-x_2)+\psi(-x-\mu t-x_2)-1+\varepsilon_0 e^{-{\kappa} t}
\end{split}
\end{equation*}
If $x$ is positive and bounded, then
\begin{equation*}
\begin{split}
  &v(t,x)\leq \psi(x-\mu t-x_2)+\psi(-\mu t-x_2)-1+\varepsilon_0 e^{-{\kappa} t}.
\end{split}
\end{equation*}
By Lemma \ref{lem9}, there exists some $k>0$, $K>0$ and $\lambda_1<0$ such that
\begin{equation*}
  k e^{\lambda_1\xi}\leq 1-\psi(\xi)\leq K e^{\lambda_1\xi},~\text{for}~ \xi\geq0.
\end{equation*}
Therefore, for $-\mu t-x_2\geq0$, we have
\begin{equation*}
\begin{split}
  &v(t,x)\leq \psi(x-\mu t-x_2)-k e^{\lambda_1(-\mu t-x_2)}+\varepsilon_0 e^{-{\kappa} t}.
\end{split}
\end{equation*}
Similarly, for $-\mu t-x_1\geq0$, we have
\begin{equation*}
\begin{split}
  &v(t,x)\geq \psi(-\mu t-x_1)+\psi(-x-\mu t-x_1)-1-\varepsilon_0 e^{-{\kappa} t}\\
  &\geq \psi(-x-\mu t-x_1)-K e^{\lambda_1(-\mu t-x_1)}+\varepsilon_0 e^{-{\kappa} t}.
\end{split}
\end{equation*}
Letting $t\rightarrow\infty$, we have
\begin{equation}
   \lim\limits_{t\rightarrow\infty}v(t,x)=\psi(\infty)=\gamma~\text{in}~L^{\infty}_{loc}(\mathbb{R}).
\end{equation}
Since $g\geq f$, by the comparison principle, we have
\begin{equation}\label{34}
  \gamma=\lim\limits_{t\rightarrow\infty}v(t,x)\leq\lim\limits_{t\rightarrow\infty}u^L(t,x)\leq\lim\limits_{t\rightarrow\infty}u^L(t,0)<\gamma,
\end{equation}
 which in turn forces a contradiction. Thus $\lim\limits_{t\rightarrow\infty}u^L(t,0)=1$ if $\lim\limits_{t\rightarrow\infty}u^L(t,0)>\beta$ and we have by \eqref{34} that
\begin{equation*}
\lim\limits_{t\rightarrow\infty}u^L(t,x)= 1 ~\text{in}~L^{\infty}_{loc}(\mathbb{R}),
\end{equation*}
and so $L\in M_1$.
\end{proof}
\begin{prop}\label{prop5}
$M_1$ is either empty or open.
\end{prop}
\begin{proof}
Suppose that $M_1\neq\emptyset$ ($t^*$ defined in Lemma \ref{prop2} is finite), and set $L^1=\inf M_1$. Then the comparison theorem gives $(L_1,\infty)\subset M_1$. It remains to prove $L^1\notin M_1$, but this follows easily by a similar argument as in the proof of Proposition \ref{prop7}.
\end{proof}

\begin{prop}\label{prop6}
If $\lim\limits_{t\rightarrow\infty}u^L(t,0)=\beta$, then $L\in T$. Furthermore, $u^L(t_n,x)$ converges to $u_{\infty}(x)$ as $t_n\rightarrow\infty$ for some $t_n$ in $L^{\infty}_{loc}(\mathbb{R})$, where $u_{\infty}(x)$ satisfies
\begin{equation}
\begin{split}
   & J*u_{\infty}-u_{\infty}-f(u_{\infty})=0, \\
   & u_{\infty}(0)=\beta.
\end{split}
\end{equation}
\end{prop}

\begin{proof}
The basic idea in the proof comes from \cite{bates1997traveling} and \cite{fife1977approach}. In what follows, for notational convenience, we will suppress the superscript $L$.
If $\lim\limits_{t\rightarrow\infty}u(t,0)=\beta$, then $0\leq \lim\limits_{t\rightarrow\infty}u(t,x)\leq\beta$ for all $x\in\mathbb{R}$. Consequently, by \eqref{64}, in this case
 \begin{equation*}
  c=\frac{\int_{\mathbb{R}} f(u)\,{\dif}u}{\int_{\mathbb{R}}(U')^2\,{\dif}\xi}=\frac{\int_0^{\beta} f(u)\,{\dif}u}{\int_{\mathbb{R}}(U')^2\,{\dif}\xi}=0.
 \end{equation*}

We know that $u$ is bounded and equicontinuous for $t\geq N$ by Lemmas \ref{max}, \ref{com} and \ref{lem5}. Let $\{t_n'\}$ be a given sequence. If there is a finite accumulation point $t_{\infty}$, then the continuity of $u$ implies that $u(t,\cdot)$ approaches $u(t_{\infty},\cdot)$ along a subsequence. So assume there is none. For any $K>0$, let $u^K(t,x)$ be the restriction of $u$ to the set $|x|\leq K$, $t\geq N$. Applying Arzel\`a-Ascoli theorem, for each $K=1,2,\cdots$, there exists a subsequence $\{t_{n,K}\}$ such that $\{u^K(t_{n,K},x)\}$ converges as $t_{n,K}\rightarrow\infty$ in $L^{\infty}[-K,K]$. We may always choose $\{t_{n,K+1}\}$ to be a subsequence of $\{t_{n,K}\}$. We then take a diagonal subsequence of $\{t_{n,K}\}$, denoted by $\{t_n\}$, such that $\{u(t_n,x)\}$ converges uniformly on each interval $[-K,K]$ to a limit $u_{\infty}(x)$,  as $t_n\rightarrow\infty$ in the $L^{\infty}$ norm.



%
%
%
%
%
%
%
%
%
%
%
%
%
Define the left truncation
\begin{equation}\label{24}
  w(t,x):=\left\{
           \begin{array}{ll}
             \eta(-x-t)u(t,x), & \hbox{$x<0$;} \cr
             1-\eta(x)(1-u(t,x)), & \hbox{$x\geq0$;}\cr
           \end{array}
         \right.
\end{equation}
where $\eta\in C^{\infty}(\mathbb{R})$ satisfies $\eta(x)=1$ for $x\leq 0$ and $\eta(x)=0$ for $x\geq1$.
Then $0\leq w(t,x)\leq1$, $w(t,x)\equiv0$ for $x\leq-t-1$, and $w(t,x)\equiv1$ for $x\geq1$. Besides, it is easy to see that $w(t,x)\geq u(t,x)$ for  $x\geq0$, and $w(t,x)\leq u(t,x)$ for  $x<0$.

We define a Lyapunov functional by
\[
  V(t):=\int_{\mathbb{R}}\bigg[\frac{1}{2}(J* w-w)w-F(w)+H(x)F(1)\bigg]\,{\dif} x,
\]
where $F(u):=\int_0^uf(s)\,{\dif} s$, $H(x)$ is the Heaviside step function. Now we are going to prove that $V$ is well-defined and bounded for all $t\geq0$. To start with, the last two terms can be splitted into three parts (note that $w(t,x)\equiv0$ for $x\leq-t-1$, and $w(t,x)\equiv1$ for $x\geq1$):
\[
  \int_{\mathbb{R}}[H(x)F(1)-F(w)]\,{\dif} x=-\int_{-t-1}^{-t}F(w)\,{\dif} x-\int_{-t}^{0}F(w)\,{\dif} x+\int_0^1(F(1)-F(w))\,{\dif} x.
\]

From now on, let $C$ be a generic positive constant which may vary from line to line. By \textbf{(H2)} and \textbf{(H3)}, there exists some constant $C$ such that $ |F(u)|=|\int_0^u[f(s)-f(0)]\,{\dif}s|\leq\max\limits_{s_\theta\in(0,s)}|f'(s_\theta)|\frac{1}{2}u^2\leq Cu^2. $ Therefore, it is easy to check that the first and the last term above are bounded. For the second term, noting that
\begin{equation*}
\begin{split}
&  \left|\int_{-t}^{0}F(w)\,{\dif} x\right|\leq C\int_{-t}^{0}w^2\,{\dif} x\leq C\int_{-t}^{0}w\,{\dif} x=C\int_{-t}^{0}u\,{\dif} x,
\end{split}
\end{equation*}
then we have by Lemma \ref{lem4} with $c=0$ that
\begin{equation}
\begin{aligned}
&u\leq U(x-x_2)+U(-x-x_2)-1+\varepsilon_0 e^{-{\kappa} t}\leq U(-x-x_2)+\varepsilon_0 e^{-{\kappa} t}.
\end{aligned}
\end{equation}
Hence
\begin{equation}
\begin{split}
   &   \left|\int_{-t}^{0}F(w)\,{\dif} x\right|\leq C\int_{-t}^{0}U(x-x_2)\,{\dif} x+Ct e^{-{\kappa} t}.
\end{split}
\end{equation}
On the other hand, since $f'(0)>0$, we can always find some positive constants $C$ and $R$ such that for $x\leq -R$,
\begin{equation*}
J*U(x)-U(x)=f(U)=f(U)-f(0)\geq CU(x).
\end{equation*}
Integrating the above inequality on $(-\infty,-R]$, we have
\[
  \begin{aligned}
    &\int_{-\infty}^{-R}U(x)\,{\dif} x\leq C\int_{-\infty}^{-R}[J*U(x)-U(x)]\,{\dif} x\\
    &=C\int_{-\infty}^{-R}\int_{-\infty}^{\infty}J(y)\big(U(x-y)-U(x)\big)\,{\dif} y{\dif} x \\
      &  =-C\int_{-\infty}^{-R}\int_{-\infty}^{\infty}J(y)\int_0^1U'(x-sy)y\,{\dif} s{\dif} y{\dif} x
          \\
      &  =-C\int_{-\infty}^{\infty}yJ(y)\int_0^1[U(-R-sy)-U(-\infty)]\,{\dif} s{\dif} y  \\
      &\leq C\int_{-\infty}^{\infty}|y|J(y)\,{\dif} y<\infty.
  \end{aligned}
\]
Thus
\begin{equation}\label{69}
  \int_{-\infty}^0U(x)\,{\dif} x<\infty.
\end{equation}
In addition, by L'hopstipal's Rule,
\begin{equation}\label{70}
 \lim\limits_{t\rightarrow\infty} t e^{-\kappa t}=\lim\limits_{t\rightarrow\infty}\dfrac{1}{ \kappa}e^{-\kappa t}=0.
\end{equation}
Therefore $|\int_{-t}^{0}F(w)\,{\dif} x|$ and hence $|\int_{\mathbb{R}}[H(x)F(1)-F(w)]\,{\dif} x|$ is bounded for all $t\geq0$.

For the first term, note that
\begin{equation*}
  \left|\int_{\mathbb{R}}(J* w-w)w\,{\dif} x\right|\leq  \int_{\mathbb{R}}|(J* w-w)w|\,{\dif} x\leq  \int_{\mathbb{R}}|J* w-w|\,{\dif} x.
\end{equation*}
Thus it is sufficient to show that $\int_{\mathbb{R}}|J* w-w|\,{\dif} x<\infty$. We split this integration with respect to the three intervals: $(-\infty,-t]$, $(-t,0]$, $(0,\infty)$ and denote the resulting integrals by $I_1$, $I_2$ and $I_3$, respectively. Recalling $w\leq u$ if $x\leq 0$, we find that
\[
\begin{aligned}
&I_1=\int_{-\infty}^{-t}|J* w-w|\,{\dif} x=\int_{-\infty}^{-t}\bigg|\int_{-\infty}^{\infty}J(x-y)(w(t,y)-w(t,x)){\dif}y\bigg|\,{\dif} x\cr
&\leq \int_{-\infty}^{-t}{\dif} x\int_{-\infty}^0J(x-y)|w(t,y)-w(t,x)|\,{\dif} y\cr
&+\int_{-\infty}^{-t}{\dif} x\int_0^{\infty}J(x-y)|w(t,y)-w(t,x)|\,{\dif} y\cr
&\leq \int_{-\infty}^{-t}{\dif} x\int_{-\infty}^0J(x-y)w(t,y)\,{\dif} y+\int_{-\infty}^{-t}{\dif} x\int_{-\infty}^0J(x-y)w(t,x)\,{\dif} y\\
&+2\int_{-\infty}^{-t}{\dif} x\int_0^{\infty}J(x-y)\,{\dif} y\cr
&\leq \int_{-\infty}^{-t}{\dif} x\int_{-t-1}^0J(x-y)w(t,y)\,{\dif} y
+\int_{-t-1}^{-t}{\dif} x\int_{-\infty}^0J(x-y)w(t,x)\,{\dif} y\cr
&+2\int_{-\infty}^{-t}{\dif} x\int_0^{\infty}J(x-y)\,{\dif} y\cr
&\leq\int_{-\infty}^{-t}{\dif} x\int_{-t-1}^0J(x-y)u(t,y)\,{\dif} y
+\int_{-t-1}^{-t}{\dif} x\int_{-\infty}^0J(x-y)u(t,x)\,{\dif} y\cr
&+2\int_{-\infty}^{-t}{\dif} x\int_0^{\infty}J(x-y)\,{\dif} y\cr
&:=I_{11}+I_{12}+I_{13}.
\end{aligned}
\]
We are now in position to apply Lemma \ref{lem4} to $u$. Therefore
\begin{equation*}
  \begin{aligned}
&I_{11}+I_{12}\leq\int_{-\infty}^{-t}{\dif} x\int_{-t-1}^0J(x-y)[U(y-x_2)+U(-y-x_2)-1+\varepsilon_0e^{-\kappa t}]\,{\dif} y\cr
&
+\int_{-t-1}^{-t}[U(x-x_2)+U(-x-x_2)-1+\varepsilon_0e^{-\kappa t}]\,{\dif} x\cr
&\leq \int_{-\infty}^{-t}{\dif} x\int_{-t-1}^0J(x-y)U(y-x_2)\,{\dif} y+(t+1)\varepsilon_0e^{-\kappa t}\cr
&
+U(-t-x_2)+\varepsilon_0e^{-\kappa t}.
   \end{aligned}
\end{equation*}
Rewriting the first term above, we get
\begin{equation*}
  \begin{split}
      & \int_{-\infty}^{-t}{\dif} x\int_{-t-1}^0J(x-y)U(y-x_2)\,{\dif} y \\
      & =\int_{-\infty}^{-t}{\dif} x\int_{-t-1}^{-\frac{t}{2}}J(x-y)U(y-x_2)\,{\dif} y+\int_{-\infty}^{-t}{\dif} x\int_{-\frac{t}{2}}^0J(x-y)U(y-x_2)\,{\dif} y\\
    &\leq \int_{-1}^{\infty}(1+y)J(y)\,{\dif}yU(-\frac{t}{2}-x_2)+\frac{t}{2}\int^{\infty}_{\frac{t}{2}}J(y)\,{\dif}y\\
&\leq(1+\int_{\mathbb{R}}|y|J(y)\,{\dif}y)U(-\frac{t}{2}-x_2)+\int^{\infty}_{\frac{t}{2}}yJ(y){\dif}y.
   \end{split}
\end{equation*}
Moreover, we have
\begin{equation*}
  \begin{split}
      & I_{13}=2\int_{-\infty}^{-t}{\dif} x\int_0^{\infty}J(x-y)\,{\dif} y \leq 2\int_{t}^{\infty}|y|J(y)\,{\dif}y.
   \end{split}
\end{equation*}
Combining those above estimates, we conclude that $I_1$ is bounded for all $t\geq0$. In fact, it is easy to see that $I_1\rightarrow 0$ as $t\rightarrow\infty$. Similarly,
\begin{equation*}
  \begin{split}
      &  I_2=\int_{-t}^0|J* w-w|\,{\dif} x\\
      & \leq \int_{-t}^0\,{\dif} x\int_{-\infty}^{0}J(x-y)|w(t,y)-w(t,x)|{\dif}y+\int_{-t}^0\,{\dif} x\int^{\infty}_{0}J(x-y)|w(t,y)-w(t,x)|{\dif}y\\
&\leq \int_{-t}^0\,{\dif} x\int_{-\infty}^{0}J(x-y)[w(t,y)+w(t,x)]{\dif}y
+2\int_{-t}^0\,{\dif} x\int^{\infty}_{0}J(x-y){\dif}y:=I_{21}+I_{22}.
  \end{split}
\end{equation*}
Noting that $w= u$ when $x\in (-t,0)$ and then applying Lemma \ref{lem4}, we have
\begin{equation*}
  \begin{split}
      &  I_{21}= \int_{-t}^0\,{\dif} x\int_{-\infty}^{0}J(x-y)u(t,y){\dif}y+\int_{-t}^0\,{\dif} x\int_{-\infty}^{0}J(x-y)u(t,x){\dif}y\\
      &\leq \int_{-t}^0\,{\dif} x\int_{-\infty}^{0}J(x-y)(U(y-x_2)+U(-y-x_2)-1+\varepsilon_0e^{-\kappa t}){\dif}y\\
&+\int_{-t}^0\,{\dif} x\int_{-\infty}^{0}J(x-y)(U(x-x_2)+U(-x-x_2)-1+\varepsilon_0e^{-\kappa t}){\dif}y\\
&\leq \int_{-\infty}^{0}U(y-x_2)\,{\dif}y+2t\varepsilon_0e^{-\kappa t}+\int_{-t}^0U(x-x_2){\dif}x,
  \end{split}
\end{equation*}
which is bounded by \eqref{69} and \eqref{70}.
Since \textbf{(H1)} implies
\begin{equation*}
I_{22}\leq 2\int^{\infty}_0|y|J(y)\,{\dif} y<\infty,
\end{equation*}
we can conclude that $I_2<\infty$. Finally, we have
\begin{equation*}
\begin{split}
&I_{3}=\int_0^1\bigg|\int_{\mathbb{R}}J(x-y)(w(t,y)-w(t,x)){\dif}y\bigg|\,{\dif} x+\int_1^\infty\bigg|\int_{\mathbb{R}}J(x-y)(w(t,y)-w(t,x)){\dif}y\bigg|\,{\dif} x\\
&\leq 2\int_0^1\int_{\mathbb{R}}J(x-y)\,\dif y\,{\dif}x+2\int_1^\infty\int_{-\infty}^0 J(x-y){\dif}y\,{\dif} x+\int_1^\infty\int_{0}^1 J(x-y){\dif}y\,{\dif} x\\
&\leq C(1+\int_{\mathbb{R}}|y|J(y)\,{\dif}y)<\infty.
\end{split}
\end{equation*}
Combining the estimates for $I_1-I_3$, we have proved that the Lyapunov functional $V[w]$ is bounded for all $t\geq0$. By Lebesgue's Theorem and the fact that $\frac{\partial w}{\partial t}$, $J*w-w$ and $f(w)$ are bounded for $t\geq0$, we can change the order of differentiation and integration to obtain
\[
  V'(t)=\int_{\mathbb{R}}(J* w-w-f(w))\frac{\partial w}{\partial t}(t,x)\,{\dif} x.
\]
Define $Q(t)$ by $Q(t):=\int_{\mathbb{R}}{(J* w-w-f(w))}^2\,{\dif} x$. We proceed to show $P(t):=V'(t)-Q(t)\rightarrow0$ as $t\rightarrow\infty$. Since $P(t)\equiv0$ if $x\in (-\infty,-t-1]\cup(-t,0)\cup [1,\infty)$, we have
\[
  \begin{aligned}
     P(t) & = \int_{-t-1}^{-t}\bigg(J* w-w-f(w)\bigg)\bigg(\frac{\partial w}{\partial t}-J* w+w+f(w)\bigg)\,{\dif} x\cr
&+\int_{0}^{1}\bigg(J* w-w-f(w)\bigg)\bigg(\frac{\partial w}{\partial t}-J* w+w+f(w)\bigg)\,{\dif} x:=P1+P2.
   \end{aligned}
\]
It follows from \textbf{(H2)}, Lemma \ref{lem4} and the boundedness of $\frac{\partial w}{\partial t}$, $J*w-w$ and $f(w)$ that
\[
  \begin{aligned}
    P_1 & \leq C\int_{-t-1}^{-t}|J* w-w-f(w)|\,{\dif} x\\
&\leq C\int_{-t-1}^{-t}|J* w-w|\,{\dif} x+C\int_{-t-1}^{-t}|f(w)|\,{\dif} x\\
&\leq I_1+C\int_{-t-1}^{-t}Cw\,{\dif} x\leq I_1+C\int_{-t-1}^{-t}u\,{\dif} x\\
&\leq I_1+C\int_{-t-1}^{-t}[U(x-x_2)+U(-x-x_2)-1+\varepsilon_0e^{-\kappa t}]\,{\dif} x\\
&\leq I_1+CU(-t-x_2)+\varepsilon_0e^{-\kappa t}\rightarrow0,~\text{as}~t\rightarrow\infty.
  \end{aligned}
\]
The definition of $I_1$ was mentioned earlier. Similarly, we can prove that $P_2$ converges to $0$ as $t\rightarrow\infty$.
%
%
%
%
Therefore
\begin{equation}\label{25}
V'(t)-Q(t)=P(t)\rightarrow0~\mbox{as~}t\rightarrow\infty.
\end{equation}
Since $Q(t)\geq0$, it follows that $\liminf\limits_{t\rightarrow\infty}V'(t)\geq0$. Note that $\liminf\limits_{t\rightarrow\infty}V'(t)>0$ implies $V(t)\rightarrow\infty$ as $t\rightarrow\infty$, which is a contradiction. Thus there exists a sequence $\{t_n\}$ with $t_n\rightarrow\infty$ such that  $V'(t_n)\rightarrow0$. Combining this and \eqref{25}, we have
\begin{equation}\label{28}
  \lim\limits_{t_n\rightarrow\infty}Q(t_n)=0.
\end{equation}
Thus there is a subsequence of $\{t_n\}$, denote by $\{t_n'\}$ such that $w(t_n',x)$ converges to a limit function $w_{\infty}$ in the $L^{\infty}$ norm. From this and \eqref{28}, for any finite interval $I$,
\begin{equation*}
  \int_I (J* w-w-f(w))^2|_{t=t_n'}\,{\dif} x\rightarrow \int_I (J* w_{\infty}-w_{\infty}-f(w_{\infty}))^2\,{\dif} x=0.
\end{equation*}
Thus $w_{\infty}$ is a stationary wave of \eqref{1}:
\begin{equation}
  J* w_{\infty}-w_{\infty}-f(w_{\infty})=0.
\end{equation}
Note that $w_{\infty}(0)=\lim\limits_{t_n'\rightarrow\infty}w(t_n',0)=\lim\limits_{t_n'\rightarrow\infty}u(t_n',0)=\beta$. Since $w=u$, if $-t\leq x\leq 0$, then $u_{\infty}(x)=w_{\infty}(x)$ for $x<0$. The proof is completed by using the symmetrical argument for $x\geq0$.
%
%
%
%
%
%
%
%
\end{proof}

\begin{rem}\label{small}
If the monotonicity assumption $1+f'(s)>0$~for all $s\in(0,1)$ in \textbf{(H2)} is violated, $u_\infty$ might be discontinuous. See Section 3 in \cite{bates1997traveling} for more details.
\end{rem}

The following proposition shows the sharpness of transition between extinction and propagation.

\begin{prop}\label{prop8}
There is only one element in $T$.
\end{prop}
\begin{proof}
Assume, to the contrary, that $L_1<L_2$ are both in $T$. Let $u_i$ be the solution of \eqref{57} with $u_i(0,x)=1_{(-L_i,L_i)}(x)$, $i=1, 2$. We then have $\lim\limits_{t\rightarrow\infty}u_1(t,0)=\beta$. Note that the equation \eqref{57} is translation invariant, thus we have $\bar{u}(t,\varepsilon)\rightarrow\beta$ as $t\rightarrow\infty$ when $\bar{u}$ solves \eqref{57} with initial condition $\bar{u}(0,x)=1_{(-L_1+\varepsilon,L_1+\varepsilon)}(x)$. But if $|\varepsilon|<L_2-L_1$, then $\bar{u}(0,x)<u_2(0,x)$ and by the comparison principle,
\[
  u_{\infty}(0)=\beta=\lim\limits_{t\rightarrow\infty}\bar{u}(t,\varepsilon)\leq \lim\limits_{t\rightarrow\infty}u_2(t,\varepsilon)=u_{\infty}(\varepsilon),
\]
which contradicts Lemmas \ref{max} and \ref{com}.
\end{proof}

\begin{proof}[Proof of Theorem 1.1]
The first conclusion can be proved by combining the proofs of Propositions \ref{prop3} and \ref{prop7}. For the third conclusion, it can
be proved using Propositions \ref{prop4} and \ref{prop5}. Finally, the second conclusion is straightforward combining Propositions \ref{prop6} and \ref{prop8}.
\end{proof}

\section{Numerical Results}
In this section, we present a numerical example to illustrate the results of Theorem \ref{main}. Set $f(u)=u(u-\alpha)(u-1)$ and $J(x)=\frac{1}{\sqrt{2\pi}}e^{-\frac{1}{2}x^2}$. It is easy to verify that $J$ satisfies \textbf{(H1)}, and $f$ satisfies \textbf{(H2)} and \textbf{(H3)} (we have by a simple calculation that $\beta=\frac{1}{3}(-\sqrt{2(2\alpha-1)(\alpha-2)}+2\alpha+2)$. Consider a two dimensional region $R=[0,T]\times[-X,X]$ with $X=120$ and $T=200$. We obtain the following uniform grid by choosing $M=48000$ and $N=400$:
  \begin{align}
     & x_i=-X+i\triangle x~\mbox{for}~i=0,\ldots,M, \\
     & t_k=k\triangle t~\mbox{for}~k=0,\ldots,N,
  \end{align}
where $\triangle x=\frac{2X}{M}$ and $\triangle t=\frac{T}{N}$. We shall use the convention of using $u_i^k$ to denote $u(t_k,x_i)$, where
$i=0,\ldots,M$, $k=0,\ldots,N$. As in \cite{bates2009numerical}, we use a standard finite difference scheme as follows:
\begin{equation}
  u_i^0=1_{[-L,L]}(x_i)~\mbox{for}~i=0,\ldots,M,
\end{equation}
\begin{equation}
  \frac{u_i^{k+1}-u_i^{k}}{\triangle t}=(J*u^k)_i-u_i^k-f(u_i^k)~\mbox{for}~0\leq i\leq M,~0\leq k\leq N-1,
\end{equation}
where
\begin{equation}
  (J*u^k)_i=\triangle x\bigg[\frac{1}{2}J(x_0-x_i)u_0^k+\sum\limits_{m=1}^{M-1}J(x_m-x_i)u_m^k+\frac{1}{2}J(x_{M}-x_i)u_{M}^k\bigg].
\end{equation}
Set $\alpha=0.4$ (in this case $\beta=\frac{2}{3}$). In Figures 1 and 2 we present the evolution of $u$ when $L=1.605$ and $L=1.610$, respectively. In Figures 3 and 4 we plot $u(t,0)$ for $0\leq t\leq 200$. As can be seen in figures 1-4, the numerical results are in accordance with our theoretical analyses in Lemma \ref{prop1} and Lemma \ref{prop2}. In both cases, $u(t,x)=u(t,-x)$ for all $(t,x)\in R$ and $u$ is decreasing in $|x|$.  Moreover, when $L=1.605$, $u(t,0)$ as a function of $t$ is non-increasing on $[0,\infty)$ ($t^*=\infty$ in Lemma \ref{prop2}) and there exists some bounded $t^*$ such that $u(t,0)$ is  non-increasing on $[0,t^*)$ and non-decreasing on $[t^*,\infty)$ when $L=1.610$. As predicted by Theorem \ref{main}, there exists some threshold value $L^*\in(1.605,1.610)$ such that, if $L<L^*$, $u\rightarrow0$ as $t\rightarrow\infty$; and if $L>L^*$, $u\rightarrow1$ on compacts as $t\rightarrow\infty$.

\begin{figure}[htbp]
\centering
\begin{minipage}[t]{0.48\textwidth}
\centering
\includegraphics[width=\textwidth]{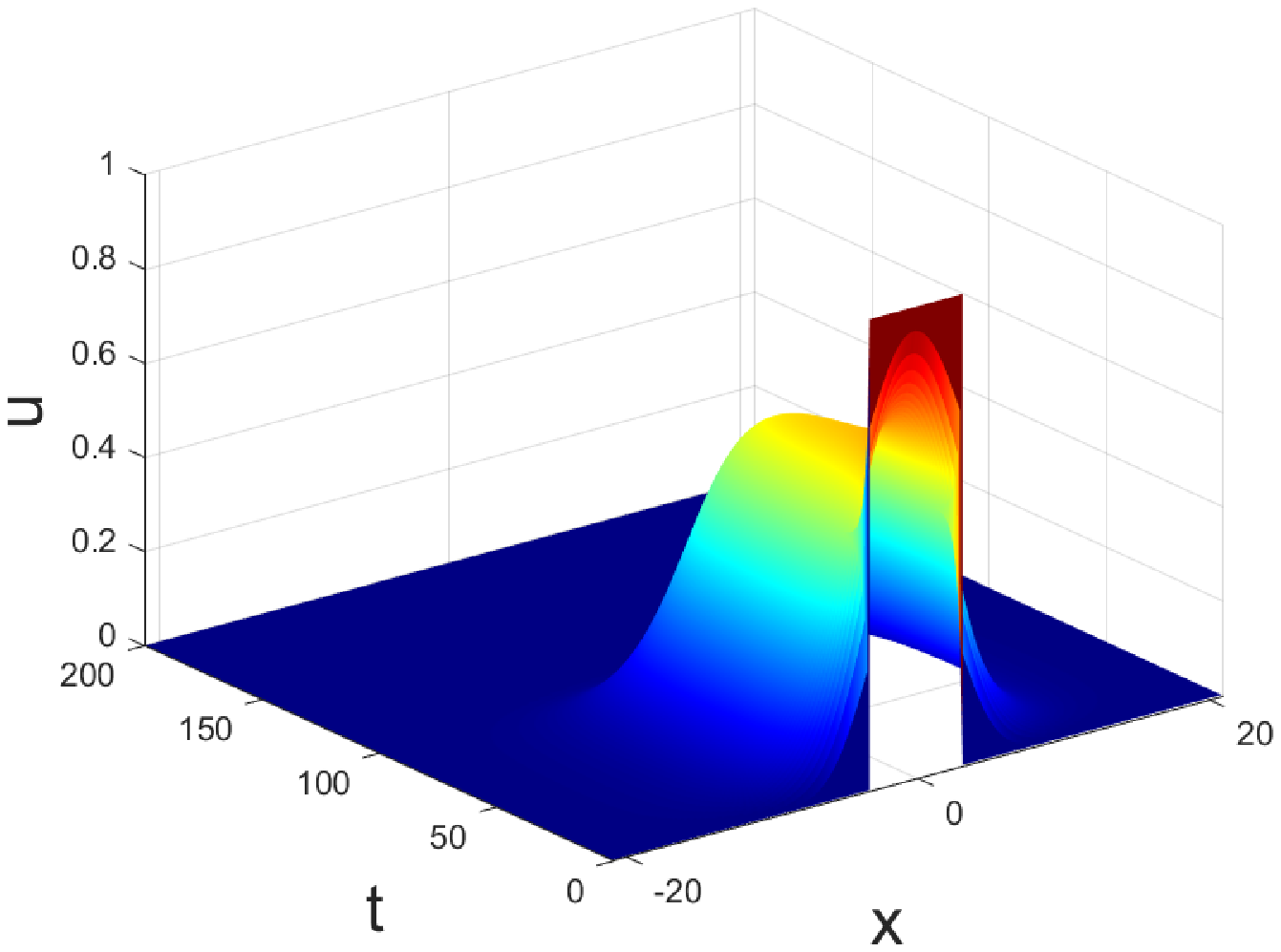}
\caption{The numerical solution when L=1.605.}\label{f1}
\end{minipage}
\begin{minipage}[t]{0.48\textwidth}
\centering
  \includegraphics[width=\textwidth]{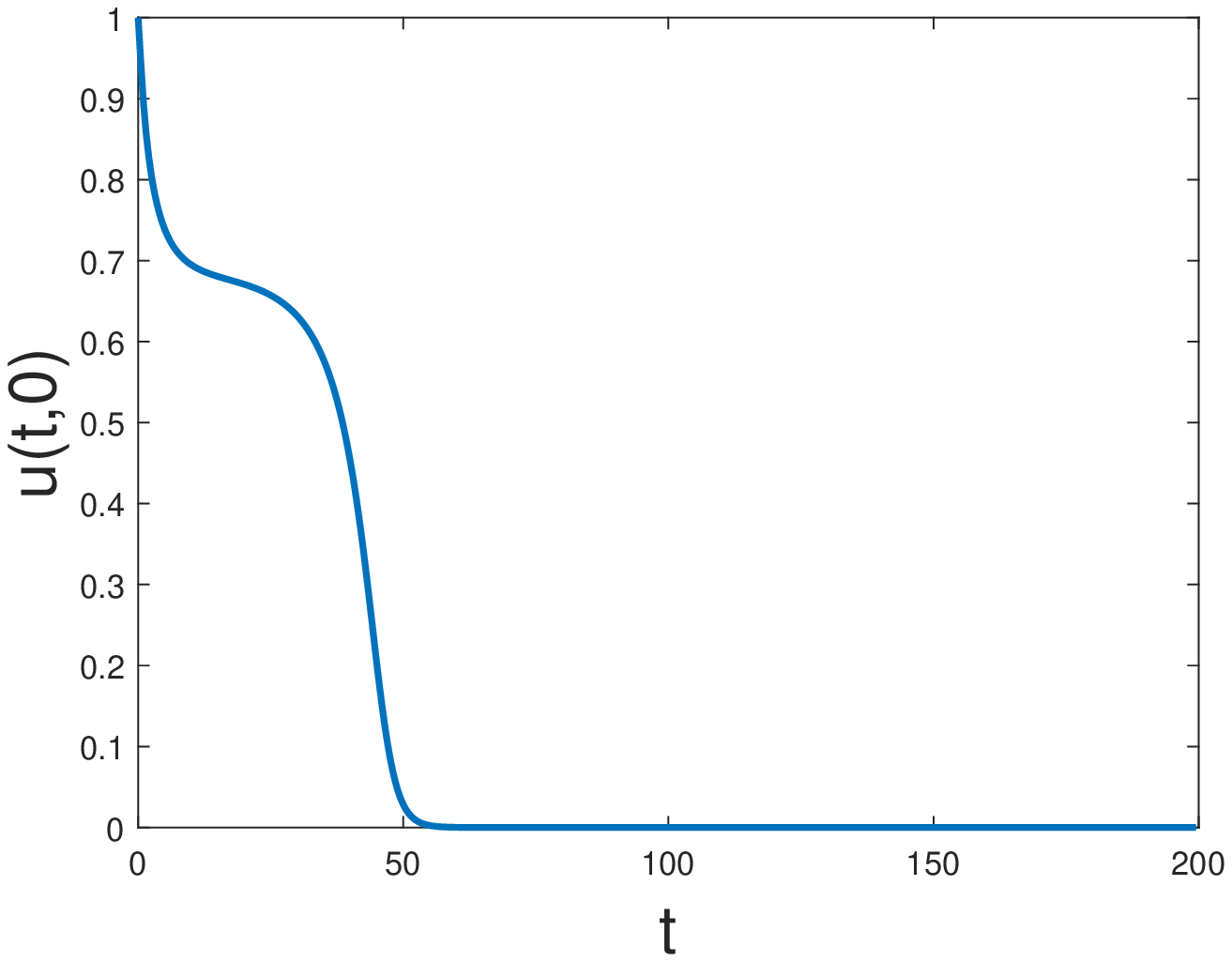}
  \caption{Plot of $u(t,0)$ when L=1.605.}\label{f2}
\end{minipage}
\end{figure}

\begin{figure}[htbp]
\centering
\begin{minipage}[t]{0.48\textwidth}
\centering
  \includegraphics[width=\textwidth]{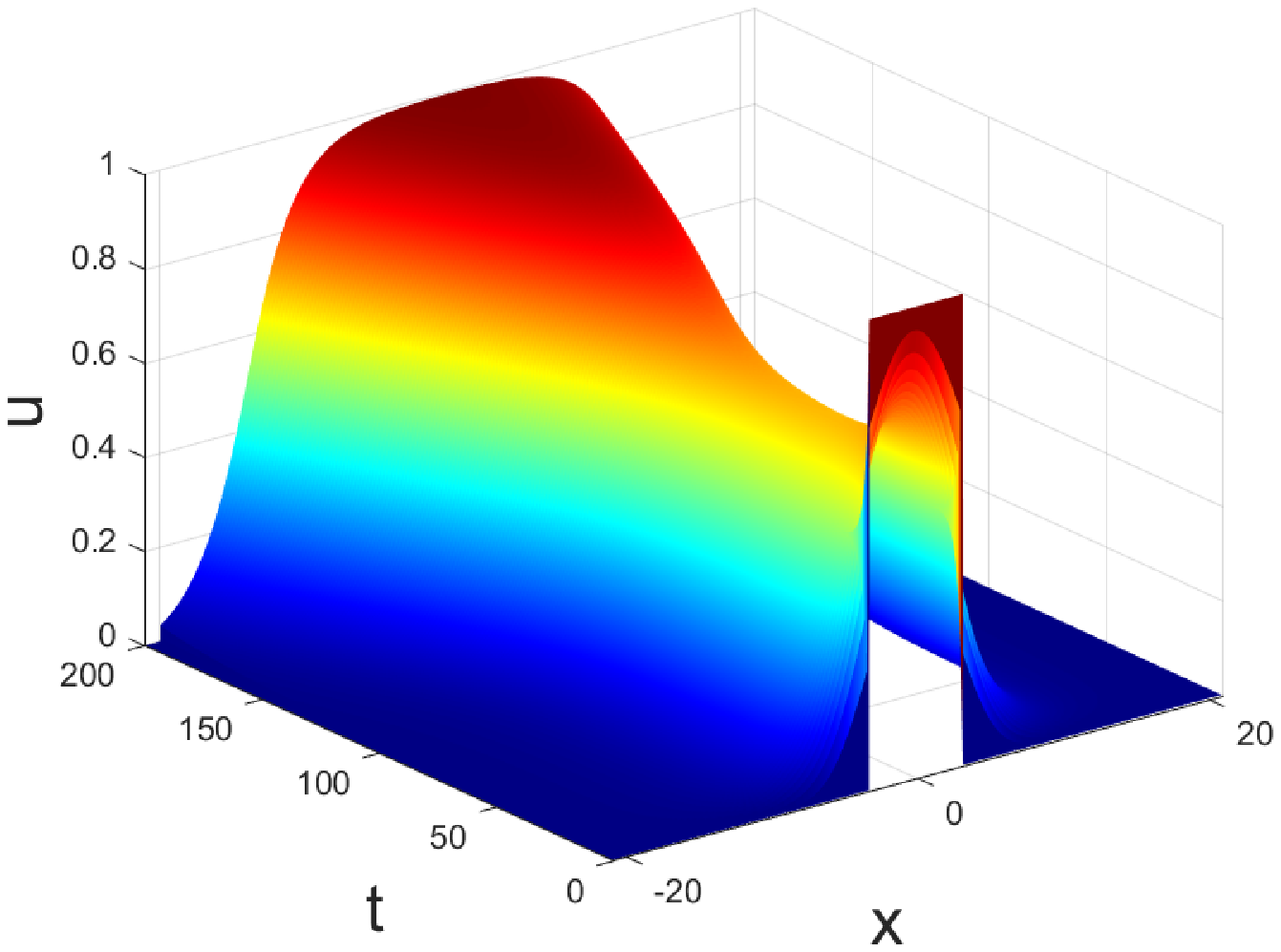}
  \caption{The numerical solution when L=1.610.}\label{f3}
\end{minipage}
\begin{minipage}[t]{0.48\textwidth}
\centering
  \includegraphics[width=\textwidth]{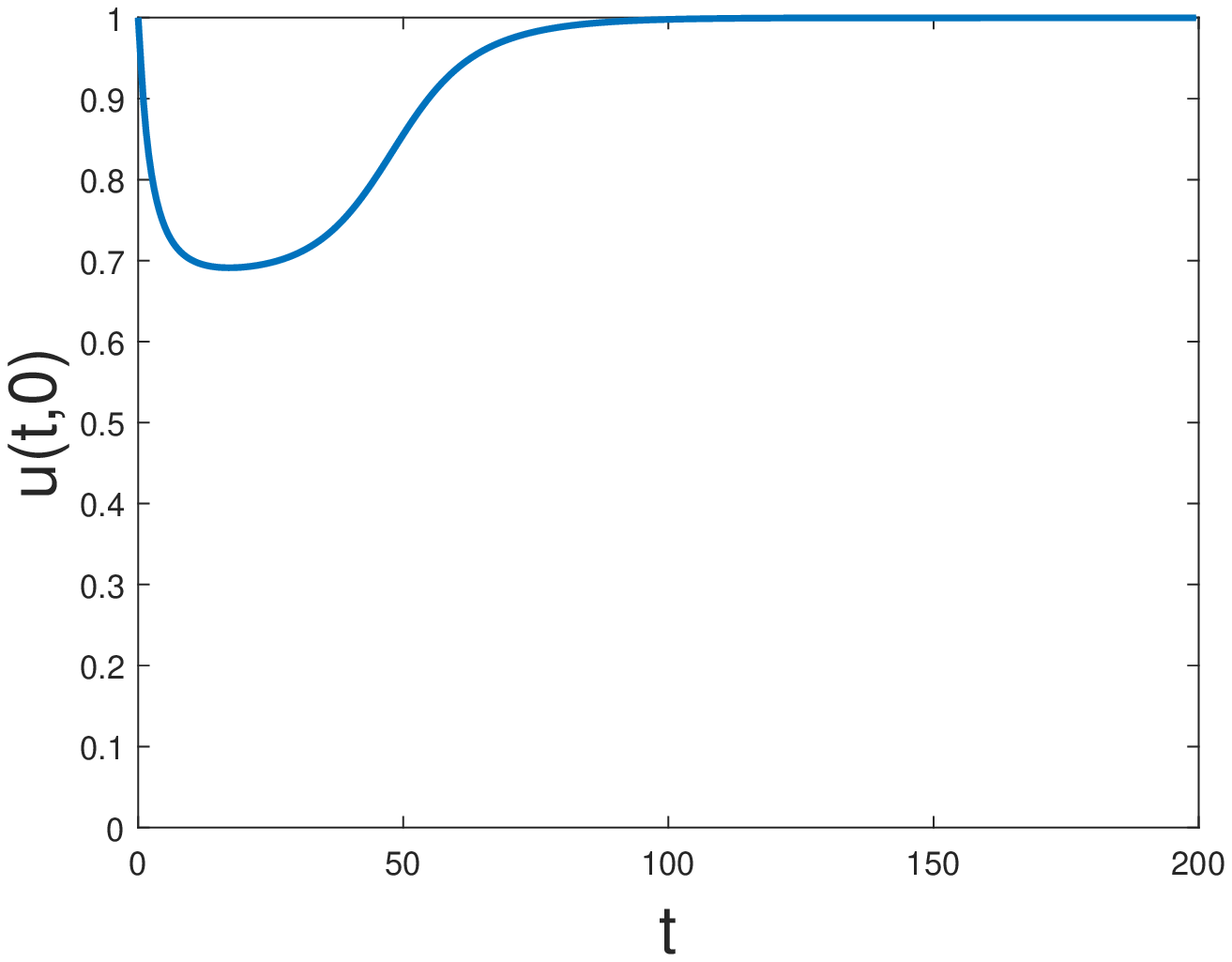}
  \caption{Plot of $u(t,0)$ when L=1.610.}\label{f4}
\end{minipage}
\end{figure}

%
%

%


\section{Concluding remarks}
In this work, we have investigated the asymptotic behaviors of solutions to nonlocal reaction diffusion equations with a one-parameter family of monotonically increasing and compactly supported initial data. We proved in Theorem \ref{main} that the solutions may either propagate (converging to 1 uniformly on compacts), become extinct (converging to 0) or converge to a nontrivial stationary wave. In particular, we have confirmed that the transition from propagation to extinction is sharp: there is only one threshold value. A few remarks are worth making to discuss possible interesting extension of the work done in this paper.

1.~Suppose that the operator $J*u-u$ in \eqref{1} is scaled by a small coefficient $\epsilon$, i.e.,
\begin{equation*}
   u_t=\epsilon(J*u-u)-f(u),~x\in\mathbb{R},~t\geq0.
\end{equation*}
Then the monotonicity assumption in \textbf{(H2)} is replaced by $\epsilon+f'(s)>0$ for $s\in(0,1)$. Thus the conclusions of Theorem \ref{main} remain true for $\epsilon$ large enough such that $\epsilon>\max_{s\in[0,1]}(-f'(s))$. The analysis for smaller $\epsilon$, however, is a much more involved task. In this case, there may exist discontinuous travelling fronts with zero speed even when $\int_0^1 f(s)\,{\dif} s\neq0$. See \eqref{64}, Remark \ref{small}, Section 3 in \cite{bates1997traveling}, and Section 5 in \cite{chen1997} for more details.

2.~In this work we have only investigated the one-dimensional equations. A natural followup question is whether similar phenomena occur in high-dimensional spaces. We expect to tackle this problem by using the comparison principle for multidimensional nonlocal reaction diffusion equations established in \cite{chen2003uniform}.

3.~For the sake of simplicity, we present the results only for a particular family of initial data $u(0,x)=1_{[-L,L]}(x).$ There is no essential difficulty with other monotone one-parameter families of continuous and bounded initial data with compact support. This has been worked out by Du and Matano \cite{du2010convergence} for one-dimensional autonomous reaction diffusion equations, and then by Pol{\'a}\v{c}ik \cite{polacik2011threshold} for multidimensional nonautonomous reaction diffusion equations.

%

\section*{Acknowledgments}
This work was completed and submitted in early 2017. Recently, we received the preprint (arXiv:2201.01512v1) by Dr. Matthieu Alfaro and colleagues, which partially overlap with ours. This work was supported in part by NSFC Grant 12071175, 11171132, 11571065 and National Research Program of China Grant 2013CB834100, Natural Science Foundation of Jilin Province (20200201253JC,  2019\\02013020JC), and Project of Science and Technology Development of Jilin Province, China (2017C028-1).

\bibliography{bibnl}
\nocite{*}
\bibliographystyle{plain}

\end{document}